\documentclass[english]{amsart}
\usepackage{fullpage}
\usepackage[T1]{fontenc}
\usepackage[utf8]{inputenc}
\usepackage{amssymb, amsthm, amsmath, amsfonts}
\usepackage{comment}

\makeatletter
\numberwithin{equation}{section}
\numberwithin{figure}{section}
  \theoremstyle{plain}
  \newtheorem*{lem*}{Lemma}
  \theoremstyle{plain}
  \newtheorem*{cor*}{Corollary}
  \theoremstyle{plain}
  \newtheorem{lem}{Lemma}
  \theoremstyle{plain}
  \newtheorem{theorem}{Theorem}
  
  \theoremstyle{remark}
  
  \theoremstyle{definition}
  \newtheorem{definition}{Definition}
\newtheorem{notation}{Notation}
\newtheorem{example}{Example}
\newtheorem{conjecture}{Conjecture}

\newcommand{\Nr}{\textnormal{Nr}}
\newcommand{\Vol}{\textnormal{Vol}}
\newcommand{\Tr}{\textnormal{Tr}}

\newcommand{\End}{\textnormal{End}}

\newcommand{\OO}{\mathcal{O}}

\newcommand{\CCC}{c}
\newcommand{\DDD}{c}

\newcommand{\R}{\mathbb R}

\newcommand{\Z}{\mathbb Z}
\newcommand{\Q}{\mathbb Q}
\newcommand{\F}{\mathbb F}
\newcommand{\C}{\mathbb C}

\newcommand{\gen}[1]{\langle #1 \rangle}
\newcommand{\norm}[1]{\Vert #1 \Vert}
\renewcommand{\comment}[1]{}

\makeatother

\usepackage{babel}

\title{Constructing supersingular elliptic curves with a given endomorphism ring}

\author{I. Chevyrev \and S. D. Galbraith}

\address{Mathematical Institute\\ University of Oxford\\ United Kingdom.}
\email{ilya.chevyrev@maths.ox.ac.uk}

\address{Mathematics Department\\ University of Auckland \\ New Zealand.}
\email{S.Galbraith@math.auckland.ac.nz}

\date{}

\begin{document}

\begin{abstract}
Let $\OO$ be a maximal order in the quaternion algebra $B_p$ over $\Q$ ramified at $p$ and $\infty$. The paper is about the computational problem: Construct a supersingular elliptic curve $E$ over $\F_p$ such that $\End(E) \cong \OO$.
We present an algorithm that solves this problem by taking gcds of the reductions modulo $p$ of Hilbert class polynomials.

New theoretical results are required to determine the complexity of our algorithm.
Our main result is that, under certain conditions on a rank three sublattice $\OO^T$ of $\OO$, the order $\OO$ is effectively characterized by the three successive minima and two other short vectors of $\OO^T$. The desired conditions turn out to hold whenever the $j$-invariant $j(E)$, of the elliptic curve with $\End(E) \cong \OO$, lies in $\F_p$. We can then prove that our algorithm terminates with running time $O(p^{1+\varepsilon})$ under the aforementioned conditions.

As a further application we present an algorithm to simultaneously match all maximal order types with their associated $j$-invariants. Our algorithm has running time $O(p^{2.5 + \varepsilon})$ operations and is more efficient than Cervi\~no's algorithm for the same problem.
\end{abstract}

\maketitle

\section{Introduction}

Let $p$ be a prime and $E$ a supersingular elliptic curve over $\F_{p^2}$. Then $\End(E)$ is a maximal order in the quaternion algebra $B_p$ ramified  exactly at $p$ and $\infty$ (all notation and definitions are explained in Section~\ref{sec:background}).
A special case of interest is when $E$ is defined over $\F_p$, in which case $\End( E )$ contains an  element $\pi$ such that $\pi^2 = -p$ (the Frobenius).
Supersingular elliptic curves have a number of algorithmic applications~\cite{ChGoLa09,Sut13}.

Ibukiyama~\cite{Ibukiyama} has given an explicit description of all maximal orders in $B_p$ that contain $\sqrt{-p}$. For example, let $p \equiv 1 \pmod{4}$ and let $\OO$ be such a maximal order in $B_p$.
Then there is a prime $q \equiv 3 \pmod{8}$ such that $( \tfrac{-q}{p} ) = -1$, and a $\Q$-algebra isomorphism $\phi : B_p \to \Q + \Q i + \Q j + \Q k$ where $i^2 = -p$, $j^2  = -q$ and $k = i j = - j i $, such that $\phi( \OO ) \cong \Z + \Z (1+j)/2 + \Z (i + k)/2 + \Z (rj + k)/q$ where $r$ is any integer such that  $q \mid (r^2 + p )$.

Consider the $\Z$-module $\OO^T = \{ 2x - \Tr( x ) : x \in \OO \}$ of rank $3$ (we discuss this object in greater detail in Section~\ref{sec:OT}). Note that $y \in \OO^T$ implies $\Tr( y ) = 0$ and so $\OO^T$ is a subset of the pure quaternions.
Fix a $\Z$-module basis $\{ \omega_1, \omega_2, \omega_3 \}$ for $\OO^T$ and consider the ternary quadratic form $Q( x, y, z ) = \Nr( x \omega_1 + y \omega_2 + z \omega_3 )$ giving a norm on $\OO^T$.
Kaneko~\cite{Kaneko} has shown, in the special case where $\sqrt{-p} \in \OO$, that there is an element $x \in \OO^T$ of norm at most $4\sqrt{p}/\sqrt{3}$.

Let $\OO'$ be another maximal order in the same quaternion algebra $B_p$ and let $Q'$ be the ternary form associated with $\OO'$.
A natural question is whether $Q$ determines $\OO$.  In other words, if $Q'$ is equivalent to $Q$ in the sense of quadratic forms then is $\OO'$ isomorphic to $\OO$?
We will show that this is the case.
Indeed, our main result (Theorem~\ref{maintheorem}) is much stronger: It states that if the forms $Q$ and $Q'$ are such that $Q'$ represents the successive minima of $Q$ (which is not the same as saying that the forms have the same successive minima), plus some other mild conditions, then $\OO \cong \OO'$, and hence $Q$ and $Q'$ are equivalent.
Schiemann~\cite{Schiemann} has shown that two ternary quadratic forms are determined up to equivalence by their theta series.
Our result may be viewed as a strong form of Schiemann's theorem in the case where both forms arise from maximal orders in the same quaternion algebra.

Our work is motivated by several computational questions about supersingular elliptic curves. One problem is, given a maximal order $\OO$ in $B_p$, to compute an elliptic curve $E$ over $\F_{p^2}$ such that $\End(E) \cong \OO$.
A second problem is to compute a list of all isomorphism classes of supersingular elliptic curves $E $ over $\F_{p^2}$ (or over $\F_p$ in a restricted case) together with a description of $\End(E)$.
To solve both problems we use Hilbert class polynomials. The main idea is that if $\OO \cong \End( E) $ and if $\OO^T$ has an element of small norm $d$ then $E$ has a ``complex multiplication'' of degree $d$ and so $j(E)$ is a root of the Hilbert class polynomial $H_{-d}(x)$.
The first problem does not seem to have been considered in the literature previously.
Cervi\~no~\cite{Cervino} has given an algorithm to solve the second problem that seems to run in $O( p^{3 + \varepsilon} )$ operations (or $O( p^{2.5 + \varepsilon} )$ in the restricted case over $\F_p$); our approach leads to a superior running time of $O( p^{2.5 + \varepsilon} )$ operations (or $O( p^{1.5 + \varepsilon} )$ in the restricted case).
%However, the main focus of our paper is the theoretical result, and further details about the algorithms and applications may be developed in future work.

\section{Background and main results}\label{sec:background}

%We recall some basic notions, and introduce some notation that we use in the statement of  our main results.
% (Theorem~\ref{maintheorem}).

Let $B_p$ be the quaternion algebra over $\Q$ ramified exactly at $p$ and at $\infty$.
%\footnote{Many of the facts we mention are true for any quaternion algebra, however our main results all use specific properties of $B_p$ and so we restrict to this case throughout the paper.} 
A general reference for many of the facts in this section is Vign{\' e}ras~\cite{Vig80}. We recall that $B_p$ is a $4$-dimensional division $\Q$-algebra containing $\Q$ with an anti-involution $x \mapsto \overline x$. Define the reduced trace $\Tr(x) = x + \overline{x}$. Then $B_p$ is equipped with the symmetric positive definite bilinear form $\Tr(x\overline y)$ and the associated positive-definite quadratic form $\Nr(x)=x\overline x$. Every element $x\in B_p$ satisfies its characteristic equation $x^2 -\Tr(x)x +\Nr(x) = 0$. We define $B_p^0$ to be the subring of $B_p$ of elements of zero trace.

We let $\OO$ and $\OO'$ be orders of $B_p$. We recall that an \emph{order} of $B_p$ is a subring of $B_p$ that contains $\Z$ and has $4$ linearly independent generators as a $\Z$-module. We recall furthermore that for all $x\in\OO$, we have $\Tr(x),\Nr(x)\in\Z$. Finally, we say that $\OO$ and $\OO'$ are of the same \emph{type} if there exists non-zero $c\in B_p$ such that $c\OO c^{-1}=\OO'$, in which case we write $\OO \sim \OO'$. 

An order $\OO$ of $B_p$ is called \emph{maximal} if it is not properly contained in any other order. Deuring showed that, associated to a maximal order $\OO$, there exists either one supersingular $j$-invariant $j(\OO)\in\F_p$, or a conjugate pair $j(\OO),\overline{j(\OO)}\in\F_{p^2}$, such that $\End(E(j(\OO))) = \End(E(\overline{j(\OO)}))= \OO$, where $E(j)$ is the unique (up to isomorphism) elliptic curve with $j$-invariant $j$.
We let the total number of maximal order types be $t_p$, the \emph{type number} of $B_p$.

If $\# \OO^* > 2$ then $j( \OO ) \in \{ 0, 1728 \}$ and the problems considered in the paper are all straightforward. More precisely, $j(\OO) =0$ if and only if there are units of (multiplicative) order $3$ and $6$, and  $j(\OO)=1728$ if and only if there is a unit of order $4$. Hence, unless otherwise stated, we assume that $\# \OO^* = 2$.

Let $V$ be any vector space over $\Q$ with a positive-definite quadratic form $\Nr$. For arbitrary vectors $v_1,v_2,\ldots,v_n \in V$, we denote by
\[
\Lambda = \gen{v_1,v_2,\ldots,v_n} := \{a_1v_1 + a_2v_2 + \ldots + a_n v_n \mid a_1,a_2,\ldots,a_n \in \Z\}
\]
the standard lattice generated by these vectors.

We say that a non-zero lattice element $x\in \Lambda$ is {\it primitive} if there do not exist $y\in\Lambda$ and $a\in\Z$ such that $ay = x$ and $a \neq \pm 1$. If $x = a_1v_1 + \ldots + a_n v_n$, then $x$ is primitive if and only if $\gcd(a_1,\ldots,a_n) = 1$. We also say that an integer $k$ is {\it represented} by $\Lambda$ if there exists $x\in\Lambda$ such that $\Nr(x) = k$, in which case we also say that $x$ {\it represents} $k$. Furthermore, we say that $x$ {\it optimally represents} $k$ if $x$ is primitive.

If $k\neq 0$, we say that $k$ is represented by $\Lambda$ with {\it multiplicity} $\theta_\Lambda(k)$, where
\[
\theta_\Lambda(k) = \frac{1}{2} \# \{(a_1,\ldots,a_n) \in \Z^n \mid \Nr(a_1v_1 + \ldots + a_n v_n) = k\},
\]
and likewise $k$ is represented optimally by $\Lambda$ with {\it optimal multiplicity} $\theta'_\Lambda(k)$, where
\[
\theta'_\Lambda(k) = \frac{1}{2} \# \{(a_1,\ldots,a_n) \in \Z^n \mid \Nr(a_1v_1 + \ldots + a_n v_n) = k, \gcd(a_1,\ldots,a_n) = 1\}.
\]
The factor $1/2 = 1/\#\OO^*$ is to avoid counting both $x$ and $-x$, since $\Nr(x)=\Nr(-x) = k$ is effectively the same representation.

Turning to the case $V = B_p$ with the quadratic form $\Nr$, for a lattice $\Lambda = \gen{v_1,v_2,v_3,v_4} \subset B_p$ we define its discriminant as $D(\Lambda) = D(v_1,v_2,v_3,v_4) = | \det(\Tr(v_i v_j)) |$ (see Section~I.4 of~\cite{Vig80}). It is a standard fact that $D(\OO)=p^2$ for a maximal order $\OO\subset   B_p$ (see, for example, Corollary~III.5.3 of~Vign{\' e}ras~\cite{Vig80}).
Note that $D( \OO ) = | \det( \Tr( v_i \overline{v_j} )) |$.

We will often think of $B_p$ simply as an inner product space and forget its algebraic structure. For example, we can find a $\Q$-basis $\{1,\tau,\rho,\tau\rho\}$ for $B_p$ such that $\tau^2 = -p, \rho^2 = -q$ and $\tau\rho=-\rho\tau$, where $q$ is a prime such that $q\equiv 3 \pmod 8$ and $\left(\frac{-p}{q}\right)=1$ (see, for example, Lemma 1.1 of Ibukiyama \cite{Ibukiyama}). Then in particular, $\Nr(a+b\tau+c\rho+d\tau\rho) = a^2 +b^2 \Nr(\tau) + c^2\Nr(\rho)+d^2\Nr(\tau\rho)$ for $a,b,c,d \in \Q$.
As such, we will embed $B_p$ into $\R^4$ by the mapping
\[
\phi: a+b\tau +c\rho+d\tau\rho \ \longmapsto  \ ae_1+b\sqrt{\Nr(\tau)}e_2+c\sqrt{\Nr(\rho)}e_3+d\sqrt{\Nr(\tau\rho)}e_4,
\]
where $e_i$ are the usual orthonormal vectors in $\R^4$. We observe that $\phi$ is indeed an isometry (the quadratic form on $\R^4$ being understood as the square of the standard Euclidean norm). We note that this is not the only standard way to represent $B_p$ (see, for example, Proposition 5.1 of Pizer~\cite{Pizer} for a different, but related representation). In particular, the above representation of $B_p$ is not the one used in the two examples of Section~\ref{examples}.

%In order to define the successive minima of a lattice $\Lambda \subset B_p$, we first do so for a lattice $L$ of $\R^m$. 
For a $n$-dimensional lattice $L$ in $\R^m$, let $\det(L)$, the determinant of $L$, be the square of the volume of $L$, i.e., if $B$ is a basis matrix for $L$ then $\det(L) := \det(BB^T) = \Vol(L)^2$. Notice that this is different to the more common definition of $\det(L)=\sqrt{\det(BB^T)} = \Vol(L)$. We say that the $n$ successive minima of $L$ are $D_1,D_2,\ldots,D_n \in \R$ such that $D_i$ is minimal such that there exist $i$ linearly independent vectors $v_1,v_2,\ldots,v_i \in L$ with $\norm{v_j}^2 \leq D_i$ for all $j \leq i$, where $\norm{\cdot}$ is the standard Euclidean norm in $\R^m$. Again we remark that our definition is the square of the more common definition where $\norm{v_j}\leq D_i$ is taken instead of $\norm{v_j}^2 \leq D_i$.

Under this notation, standard lattice bounds show that there is a minimal constant $\gamma_n$ (called the $n$-th Hermite constant) such that
\begin{equation}\label{bounds}
\det(L) \leq \prod_{i=1}^n D_i \leq \gamma_n^n \det(L).
\end{equation}
Again, this is the square of the usual equation $\prod_i \norm{v_j} \leq \gamma_n^{n/2} \Vol(L)$.
It is known that $\gamma_2^2 = 4/3$ and $\gamma_3^3 = 2$ (see Section XI.5 and XI.6 of Siegel~\cite{Siegel}).

Now for any lattice $\Lambda \subset B_p$, the determinant, volume and successive minima of $\Lambda$ are defined to be those of $\phi(\Lambda) \subset \R^4$, where $\phi:B_p \mapsto \R^4$ is the embedding described above. We note that for a $4$-dimensional lattice $\Lambda \subset B_p$, we have  
\begin{equation}\label{eq:DLambda}
   D(\Lambda) = 16\det( \phi( \Lambda))
\end{equation}
since $\Tr(x \overline y) = 2 \phi(x) \phi(y)^T$.

One goal of this paper is to give sufficient conditions under which the elements of small norm of a maximal order $\OO$ of $B_p$ characterise its type. 
The first theorem is that the successive minima of the lattice $\OO^T$ determine the type of the order.

\begin{theorem}\label{theorem1}
Let $\OO$ and $\OO'$ be two maximal orders of $B_p$.
Let $\OO^T$ and $\OO'^T$ have the same successive minima $ D_1 \le D_2 \le D_3$.
Assume moreover that $ D_1D_2 < 16p/3$ and that $p$ is sufficiently large.
Then $\OO$ and $\OO'$ are of the same type. 
\end{theorem}

Our main result is a stronger statement as it does not require both orders to give lattices with the same successive minima. It is this result we need later for our algorithmic application.

\begin{theorem}\label{maintheorem}
Let $p > 286$ and $\OO$, $\OO'$ be two maximal orders of $B_p$.
Let $D_1$, $D_2$ and $D_3$ be the successive minima of $\OO^T$ and let $x, y \in \OO^T$ be such that $\Nr( x ) = D_1$ and $\Nr(y) = D_2$.
Suppose that $D_1$, $D_2$, $\Nr(x+y)$, $\Nr(x-y)$ and $D_3$ are all represented optimally in $\OO'^T$ and that $\theta'_{\OO^T}(D_3) \leq \theta'_{\OO'^T}(D_3)$. Assume moreover that
\begin{equation}\label{maininequality}
D_1D_2 < \frac{16}{3}p.
\end{equation}
Then $\OO$ and $\OO'$ are of the same type. 
\end{theorem}

We demonstrate the proof of Theorem~\ref{theorem1} and~\ref{maintheorem} in Section~\ref{sec:thm1Proof} and Appendix~\ref{sec:thmMainProof} respectively.
We remark that $D_1D_2 < 16p/3$ may seem very restrictive, however Lemma~\ref{lem:conditions} demonstrates a set of cases when this condition holds.

\begin{lem} \label{lem:conditions}
Let $\OO$ be a maximal order in $B_p$ and $D_1$ and $D_2$ the first two successive minima of $\OO^T$. If $\OO$ contains an element $\pi$ such that $\pi^2 = -p$ (or equivalently, if $j(\OO) \in \F_p$), then $D_1D_2 < 16p/3$.
\end{lem}

\begin{proof}
When $j(\OO) \in\F_p$, Kaneko proves (see the proof of Theorem 1 of~\cite{Kaneko} on pages 851--852) that there exists a $2$-dimensional sublattice $\Lambda$ of $\OO^T$ with determinant $\det(\Lambda)=4p$.
Let $d_1$ and $ d_2$ be the two first successive minima of $\Lambda$. %As mentioned in Section~\ref{sec:background},
Using the second Hermite constant $\gamma_2^2 = 4/3$ in~\eqref{bounds}, we obtain that $4p \leq d_1d_2 < 16p/3$ (the second inequality is strict since $d_1d_2$ is an integer and the case $p=3$ is trivial). Finally, since $D_i \leq d_i$ for $i=1,2$, it follows that $D_1D_2 < 16p/3$.
\end{proof}

Elkies showed that $D_1 \le 2 p^{2/3}$ for any maximal order in $B_p$.  Yang~\cite{Yang} has shown that Elkies' result is the best possible.
%Kohel~\cite{Koh14} has shown that if $D_1D_2 < 16p/3$ then there exists an element $\pi \in \OO$ such that $\pi^2 = -p$, which proves the converse of Lemma~\ref{lem:conditions}.
% can be made an if and only if statement.
% Referee asks for proof, but it is not *our* proof.

% Hence, the condition $D_1 D_2 <  16p/3$ is certainly not satisfied for all maximal orders. Yang has also shown that Kaneko's result is best possible, however this does not imply the converse of Lemma~\ref{lem:conditions}.   Overall, it seems that our theorem applies to maximal orders corresponding to supersingular elliptic curves for which there is a low degree isogeny to a supersingular elliptic curve over $\F_p$.

\section{The lattice $\OO^T$ and its properties}  \label{sec:OT}

\begin{definition} \label{defn:OT}
For an order $\OO\subset B_p$, we define $\OO^T =\{2x-\Tr(x)\mid x\in\OO\}$. 
\end{definition}

We remark that $\OO^T$ is a sublattice of $\OO\cap   B_p^0$, and this inclusion is strict.
The set $\OO^T$ is called the ``Gross lattice'' by some authors (see Yang~\cite{Yang} and Kane~\cite{Kane}).

If we have $\OO=\gen{1,u_1,u_2,u_3}$ for $u_1,u_2,u_3 \in B_p$ and let $v_i = 2u_i-\Tr(u_i)$, it follows immediately that $\OO^T = \gen{v_1,v_2,v_3}$.
As already noted, the discriminant of a maximal order $\OO\in   B_p$ is $p^2$. The following basic result on the determinant of $\OO^T$ follows directly from these two remarks and is a special case of Corollary~71 of Kohel~\cite{Kohel} with $\alpha = 1$.

\begin{lem} \label{determinant}
Let $\OO$ be a maximal order of $B_p$. Then $\det(\OO^T)=4p^2$.
\end{lem}

The following easy lemma allows us to characterize the conjugacy classes of $B_p$. For any $x,y\in B_p$, we write $x\sim y$ if there exists non-zero $c\in  B_p$ such that $cxc^{-1}=y$. Likewise for lattices $\Lambda, \Lambda' \subset B_p$ we write $\Lambda \sim \Lambda'$ if there exists non-zero $c\in  B_p$ such that $c\Lambda c^{-1}=\Lambda'$.

\begin{lem} \label{cond}
Let $x,y\in B_p$. Then $x\sim y$ if and only if $\Tr(x)=\Tr(y)$ and $\Nr(x)=\Nr(y)$.
\end{lem}

%
%\begin{proof}
%This follows from the Skolem-Noether Theorem, see Theorem~I.2.1 of~Vign{\' e}ras~\cite{Vig80} or Theorem 5 (on page 10) of Eichler \cite{Eichler} (note that  Eichler calls it Wedderburn's Theorem).
%\end{proof}

If $\OO^T =  \gen{v_1,v_2,v_3}$ as above, it is not difficult to see that $\OO = \{ x \in 1/2 \langle 1 , \OO^T \rangle :  \Nr(x) \in \Z \}$. From this observation we obtain the following lemma which characterizes $\OO$ in terms of $\OO^T$.

\begin{lem} \label{Lem}
Two orders $\OO,\OO' \subset B_p$ are of the same type if and only if $\OO^T\sim\OO'^T$.
\end{lem}

\begin{proof}
It is clear that if $c\OO c^{-1}=\OO'$, then $c\OO^Tc^{-1}=\OO'^T$. Conversely, assume that $c\OO^T c^{-1}=\OO'^T$. By conjugating $\OO$ by $c$, we see it suffices only to prove that if $\OO^T=\OO'^T$, then $\OO$ and $\OO'$ are of the same type. But from the above observation, if $\OO^T=\OO'^T$, then $ \langle 1, \OO^T \rangle =  \langle 1, \OO'^T \rangle$ and so in fact we obtain $\OO=\OO'$.
\end{proof}

We now make some remarks about lattices generated by pairs of elements $x, y \in \OO^T$.
Let $x, y \in \OO^T$ be such that $\gen{ x, y }$ is a rank 2 lattice.
Define the $2$-dimensional subspace
\begin{equation}\label{perp}
\gen{x,y}^\perp = \{v \in B_p \mid \Tr(v\overline x) = \Tr(v\overline y) = 0\}.
\end{equation}
As $x,y$ have zero trace, we see that $\Q \subset \gen{x,y}^\perp$, and so we can suppose $\gen{x,y}^\perp$ has $\Q$-basis $\{1,w\}$ with $\Tr(w)=0$.

\begin{lem} \label{mult}
Let $x, y \in \OO^T$. It then holds that $w = 2 xy -  \Tr( xy) \in \OO^T \cap \gen{x,y}^\perp$, where $\gen{x,y}^\perp$ is defined in equation~\eqref{perp}.
\end{lem}

\begin{proof}
Clearly $w$ has trace zero.
We observe that $\Tr(xy\overline x) = \Tr(xy\overline y) = 0$ since both $x$ and $y$ have zero trace. So we have $xy \in\gen{x,y}^\perp$, and since $\Q \subset \gen{x,y}^\perp$, it follows that indeed $2xy -  \Tr(xy)\in\gen{x,y}^\perp$.
\end{proof}

Let $D_1 = \Nr( x ) $, $D_2 = \Nr(y ) $ and $L = \gen{ x, y }$.
Writing $T = \Tr(x \overline y) = x \overline y + y \overline x = - (xy + yx)$ we have that the lattice $L$ has determinant $D_1 D_2 - (T/2)^2 = ( 4 D_1 D_2 - T^2 )/4$.
Write $w = 2 x \overline y - T = x \overline y - y \overline x $.
Then, by Lemma~\ref{mult}, $w \in \OO^T \cap \gen{x,y}^\perp$.
An immediate calculation gives $\Nr( w ) = 4 D_1 D_2 - T^2$.
Hence, the determinant of $\gen{ x, y, w }$ and $\gen{ 1, x, y, w }$ is $( 4 D_1 D_2 - T^2)^2/4$. The discriminant of the order $\gen{ 1, x, y, w }$ is thus $4(4 D_1 D_2 - T^2)^2$, and since $\gen{ 1, x, y, w } \subseteq \OO$, we have $p^2 \mid (4 D_1 D_2 - T^2)^2$ and so 
\begin{equation} \label{eq:important-condition}
   p \mid (4 D_1 D_2 - T^2).
\end{equation}
(This argument appears in Kaneko~\cite{Kaneko}.)

For an integer $D < 0$ ($D\equiv 0$ or $1 \pmod 4$), we consider the imaginary quadratic order $\OO_{D} := \Z[\frac{1}{2}(D+\sqrt{D})]$ of discriminant $D$. An embedding $i:\OO_{D} \mapsto \OO$ is called {\it optimal} if $(\Q\otimes i(\OO_{D})) \cap \OO = i(\OO_{D})$. By a straightforward argument (see, for example, the beginning of Section 3 of Elkies et al.~\cite{Elkies}), we see that there is a bijection between primitive elements of $\OO^T$ and optimal embeddings in the following sense: for every optimal representation of $|D|$ in $\OO^T$ by a primitive element $x\in\OO^T$, there is a unique optimal embedding $i: \OO_{D} \mapsto \OO$ such that $i(\sqrt{D}) = x$, and vice versa. Hence, whenever we talk of an optimal representation or primitive element, we will always associate to it the corresponding optimal embedding.

\section{Proof of Theorem~\ref{theorem1}}\label{sec:thm1Proof}

We remark first that when $p$ is small, all maximal orders of $B_p$ can be found feasibly through an exhaustive search, and so this case is easily handled for both Theorems~\ref{theorem1} and~\ref{maintheorem}. It will furthermore turn out that we require bounds like $p > 168$ or $p > 286$ for some technical lemmas.
Hence, we introduce the following notation which will be used throughout the rest of the paper.

\begin{notation} \label{defn:notation}
Let $p>286$ be a prime and $\OO$ and $\OO'$ two maximal orders in $B_p$.
Let $\OO^T$ and $\OO'^T$ be as in Definition~\ref{defn:OT}.
Let $D_1, D_2, D_3$ (respectively, $D_1', D_2', D_3'$) be the successive minima of $\OO^T$ (respectively, $\OO'^T$).
Denote by $x, y, z\in \OO^T$ (respectively, $x', y', z' \in \OO'^T$) elements such that $D_1=\Nr(x), D_2 =\Nr(y), D_3 = \Nr(z)$ (respectively, $D_1' =\Nr(x'), D_2'=\Nr(y'), D_3'=\Nr(z')$).
\end{notation}

Before describing the general strategy of the proof, we remove a small number of trivial cases when $D_1$ is small. We recall that the number of different types of maximal orders of $B_p$ containing an optimal embedding of the imaginary quadratic order $\OO_{D}$ is bounded above by $h_{D}$, the class number of $\OO_{D}$ (we refer to Theorem~\ref{multroot} of Section~\ref{Algorithm} for a more detailed result). However it is known that $h_{D} = 1$ for all discriminants $-15 < D < 0$. We thus obtain the following result, relevant for both Theorems.

\begin{lem} \label{lem:trivialcase}
Let $-15 < D < 0$. If $\OO$ and $\OO'$ are maximal orders of $B_p$ which both optimally represent $|D|$, then $\OO$ and $\OO'$ are of the same type.
\end{lem}

Unless otherwise stated, we will always impose the conditions:
\begin{equation}\label{mainconditions}
  D_1D_2 < \frac{16}{3}p,\textnormal{ } 15 \leq D_1, \textnormal{ and } 286 < p.
\end{equation}
We further remark that in the setting of Theorems~\ref{theorem1} and~\ref{maintheorem}, where $\OO'^T$ optimally represents the successive minima of $\OO^T$, 
it trivially holds that
\begin{equation}\label{D2ineq}
  D_1' \leq D_1 \textnormal{ and } D_2' \leq D_2.
\end{equation}

We now describe the general strategy of the proof of Theorem~\ref{theorem1}. 
The goal is to show that $\OO$ and $\OO'$ are of the same type, which will follow from showing that $\OO^T$ and $\OO'^T$ are conjugate.
The first step is to take appropriate sublattices $\gen{ x, y }$ in $\OO^T$ and $\gen{x',y'}$ in $\OO'^T$ and then to show that $\gen{ x, y }$ and $\gen{x',y'}$ are isometric.
The final stage of the proof is to extend to the full lattices $\OO^T$ and $\OO'^T$.

\subsection{Proving that $\gen{ x, y }$ and $\gen{x',y'}$ are isometric}

Let $x, y \in \OO^T$ and $x', y' \in \OO'^T$ be as in Notation~\ref{defn:notation}, and recall that $D_1 = D_1'$ and $D_2 = D_2'$ in the case of Theorem~\ref{theorem1}.
To show that $\gen{x,y}$ and $\gen{x',y'}$ are isometric it suffices to show that $\Tr( x y )  = \Tr( x' y' )$.  This follows from equation~(\ref{eq:important-condition}), that $p$ divides $4D_1 D_2 - T^2$, where $T = \Tr( x \overline{y} )$.

\begin{lem}
Let notation be as above and suppose $p > 128$. Then $\Tr( x \overline{y} ) = \Tr( x' \overline{y'} )$.
\end{lem}

\begin{proof}
We know that $0 < D_1 D_2 < 16p/3$ and $0 \le T^2 \le 4 \Nr( x) \Nr(y) \le 4 D_1 D_2$, and similarly for $D_1', D_2', T'$.  Hence, $0 \le 4D_1 D_2 - T^2 \le 4 D_1 D_2 < 64 p/3 < 22p$ and $|T| < \sqrt{ 64 p / 3} < 4.7 \sqrt{p}$.

We also know that $ 4D_1 D_2 - T^2 \equiv 4D_1 D_2 - T'^2 \equiv 0 \pmod{p}$.
Further, there are at most two solutions modulo $p$ to $T^2 \equiv  4 D_1 D_2 \pmod{p}$, and so all possible values for $T' = \Tr( x' y')$ are of the form $T' = \pm T + k p$ for some integer $k$.
Now, $0 \le 4 D_1 D_2 - T'^2 \le 4 D_1 D_2 < 22p$, and
\[
   4 D_1 D_2 - T'^2 = (4 D_1 D_2 - T^2) \mp 2Tkp - k^2p^2.
\]
For $p > 128$ and $|k| \geq 1$ we remark that $|\mp 2Tkp - k^2p^2| \geq p(p - 2|T|) > p(p - 9.4\sqrt{p}) > 22p$. Thus $k=0$ and so $T' = \pm T$. Changing the sign of $y'$, if necessary, gives the result.
\end{proof}

We deduce that $\gen{x,y}$ and $\gen{x',y'}$, are isometric.  Hence, as shown in Lemma~\ref{sim} below, we can conjugate so that $x' = x$ and $y' = y$.

\begin{lem} \label{sim}
Let $\OO,\OO'\subset B_p$ be two orders. For any elements $x,y \in \OO^T$ and 
$x',y' \in \OO^T$ such that $x\sim x'$, $y\sim y'$ and $x+y\sim x'+y'$ it holds 
that $\gen{x,y} \sim \gen{x',y'}$, i.e., there exists non-zero $c \in B_p$ such that $c\gen{x,y}c^{-1}=\gen{x',y'}$.
\end{lem}

\begin{proof}
As $\Tr(\OO^T)=\Tr(\OO'^T)=0$, for all $r\in\OO^T$ and $r'\in\OO'^T$, it holds that $r\sim r'$ if and only if $\Nr(r)=\Nr(r')$ by Lemma~\ref{cond}. It follows that
\[
\Nr(x')+\Nr(y') + \Tr(x'\overline{y'}) = \Nr(x'+y') = \Nr(x+y)=\Nr(x)+\Nr(y)+\Tr(x\overline{y}),
\]
and we obtain $\Tr(x\overline{y})=\Tr(x'\overline{y'})$.

We recall that for any $u,v\in  B_p$, we have 
\[
  uv +vu =    \Tr(u)v+\Tr(v)u + \Tr(uv) - \Tr(u)\Tr(v).
\]
From this, it follows that $\gen{1,x,y,xy}$ and $\gen{1,x',y',x'y'}$ are both rings (just check that the product of any two generators is in the lattice), and hence they are both orders. Furthermore, since $\overline{x}=-x$, $\overline{y}=-y$ and $\Tr(x\overline{y})=\Tr(x'\overline{y'})$, we obtain that these orders are isomorphic under the natural mapping $\psi: a+bx+cy+dxy \mapsto a+bx'+cy'+dx'y'$. Since all isomorphisms of orders come from conjugation, we know that there exists non-zero $c\in  B_p$ such that $c\gen{1,x,y,xy}c^{-1}=\gen{1,x',y',x'y'}$. The lemma follows.
\end{proof}

\subsection{Completing the proof}

We now have $\OO^T = \gen{x,y,z}$ and $\OO'^T = \gen{x,y,z'}$ with $\Nr( z ) = \Nr( z' ) = D_3$.
It remains to prove that $\OO^T $ and $\OO'^T$ are equal.

We have the following result for any ternary lattice.

\begin{lem} \label{ternarylattice}
Let $L$ be a lattice of dimension $3$ endowed with a norm $\norm{\cdot}$. Let $x,y,z \in L$ and assume that $D_1:= \norm{x}^2, D_2 := \norm{y}^2$ and $D_3 := \norm{z}^2$ are the successive minima of $L$. Then $L = \gen{x,y,z}$ and (recalling that $\det(L) = \Vol(L)^2$)
\[
    \det(L) \leq D_1D_2D_3 \leq 2\det(L).
\]
\end{lem}

\begin{proof}
As mentioned in Section~\ref{sec:background}, the third Hermite constant $\gamma_3$ is given by $\gamma_3^3 = 2$. The desired inequality follows immediately from~\eqref{bounds}.

To deduce that $L = \gen{x,y,z}$, we observe that the volume of a sublattice $L' \subseteq L$ is always a multiple of the volume of $L$. Furthermore $\Vol(L) = \Vol(L')$ if and only if $L = L'$. Hence if $\gen{x,y,z} \neq L$, then $\Vol(\gen{x,y,z}) \geq 2\Vol(L)$, and so again by~\eqref{bounds}, we have
\[
D_1D_2D_3 \geq \det(\gen{x,y,z}) \geq 4\det(L),
\]
which contradicts $D_1D_2D_3 \leq 2\det(L)$. We conclude that $L = \gen{x,y,z}$ as claimed. 
\end{proof}

Lemma~\ref{ternarylattice} allows us to conclude that $\OO^T = \gen{x,y,z}$ and $\OO'^T = \gen{x',y',z'}$, and, in conjunction with Lemma~\ref{determinant}, that
\begin{equation}\label{basic}
4p^2 \leq D_1D_2D_3, D_1'D_2'D_3' \leq 8p^2.
\end{equation}

\begin{lem} \label{simorders}
Let notation be as in Notation~\ref{defn:notation}. Suppose that $\OO^T = \gen{x,y,z}$ and $\OO'^T = \gen{x,y,z'}$ with $\Nr( z ) = \Nr( z' ) = D_3$.
Then $z= \pm z'$ (from which it follows that $\OO^T = \OO'^T$) provided that
\begin{equation}\label{Sim1}
D_1D_2 < \frac{16}{3}p,
\end{equation}
\begin{equation}\label{Sim2}
15\leq D_1, \text{ and }
\end{equation} 
\begin{equation}\label{Sim3}
168 < p.
\end{equation}
\end{lem}

\begin{proof}
Recall from equation~\eqref{perp} the $2$-dimensional subspace 
\begin{equation}
   \gen{x,y}^\perp := \{v \in B_p \mid \Tr(v\overline x) = \Tr(v\overline y) = 0\}.
\end{equation}
As $x,y$ have zero trace, we see that $\Q \subset \gen{x,y}^\perp$, and so we can suppose $\gen{x,y}^\perp$ has $\Q$-basis $\{1,v\}$ with $\Tr(v)=0$. Let $u \in \gen{x,y}^\perp$ be the projection of $z$ onto $\gen{x,y}^\perp$ (that is, $u = \Tr(z\overline v) v / (2\Nr(v))$). Similarly, let $u'$ be the projection of $z'$ onto $\gen{x,y}^\perp$. We remark that $u, u' \in B_p^0$.

Now, (recalling that the determinant is the square of the volume of a lattice)
\begin{equation}\label{projections}
  \det(\gen{x,y}) \Nr(u)=\det(\OO^T)=\det(\OO'^T)=\det(\gen{x,y})\Nr(u').
\end{equation}
Since $u,u'\in\gen{v}$, it follows that $u'=\pm u$, so, replacing $z'$ by $-z'$ if necessary, we may assume $u' = u$. Write $z = ( \alpha x + \beta y) + u$ for some $\alpha, \beta \in \Q$.

Let $s = 2xy- \Tr(xy)$, which by Lemma~\ref{mult}, lies in $\OO^T \cap \gen{x,y}^\perp$ and in $\OO'^T \cap \gen{x,y}^\perp$. Hence there exist $a,b,c,a',b',c'\in\Z$ such that $s=ax+by+cz$ and $s=a'x+b'y+c'z'$.

Since $s\in \gen{x,y}^\perp \cap \OO^T$, and $u$ is the projection of $z$ and $z'$ onto $\gen{x,y}^\perp$, it holds that $s= c u = c'u$, which implies $c=c'$. 
Furthermore, we have that
\begin{equation}\label{eq1}
\Nr(ax+by)=\Nr(s-cz)=\Nr(s)+c^2\Nr(z)-c\Tr(s\overline{z}) \text{ and}
\end{equation}
\begin{equation}\label{eq2}
\Nr(a'x+b'y)=\Nr(s-cz')=\Nr(s)+c^2\Nr(z')-c\Tr(s\overline{z'}).
\end{equation}

Since the projections of $z$ and $z'$ onto $\gen{x,y}^\perp$ are equal, we obtain $\Tr(s\overline{z})=\Tr(s\overline{z'})$. We also recall that $\Nr(z)=D_3=\Nr(z')$. Together with \eqref{eq1} and \eqref{eq2}, this implies that
\begin{equation}\label{axby}
\Nr(ax+by)=\Nr(a'x+b'y).
\end{equation}

We will now show that $\Nr(ax+by)$ cannot be too large and then apply Theorem $2'$ of \cite{Kaneko} to conclude that $ax+by=\pm (a'x+b'y)$. Recall that $u=- \alpha x - \beta y +  z$, for some $\alpha,\beta \in\Q$.
We claim that the closest element to $\alpha x +\beta y$ in the lattice $\gen{x,y}$ is $0$.
Indeed, let $k \in \gen{x,y}$ be the closest lattice element to $\alpha x +\beta y$. Then $\Nr(\alpha x + \beta y - k) \leq \Nr(\alpha x + \beta y)$. On the other hand, we have that
\[
\Nr(-z-k) = \Nr(u) + \Nr(\alpha x + \beta y - k) \geq \Nr(z) = \Nr(u) + \Nr(\alpha x + \beta y),
\]
where the inequality holds since $-z-k$ is outside $\gen{x,y}$ and $z$ represents the third successive minimum of $\OO^T$. Thus $\Nr(\alpha x + \beta y - k) = \Nr(\alpha x + \beta y)$, and hence $0$ is the closest element to $\alpha x +\beta y$ in the lattice $\gen{x,y}$ as claimed.

It is well known that the covering radius $\rho(\Lambda)$ of a lattice $\Lambda$ is always bounded by $\rho(\Lambda) \leq \sigma(\Lambda)/2$, where $\sigma(\Lambda)$ is the length of the diagonal of the orthogonal parallelepiped of $\Lambda$ (see, for example, Theorem~7.9, page~138 of Micciancio and Goldwasser~\cite{Goldwasser}). As a result, we have that
\[
\Nr(\alpha x+\beta y) \leq \rho(\gen{x,y})^2 \leq \frac{1}{4}\sigma(\gen{x,y})^2 \leq \frac{1}{4}(D_1+D_2).
\]
Since $s=cu$, it holds that $a=c\alpha$ and $b=c\beta$, and so
\begin{equation}\label{boundbyc}
\Nr(a x + b y)=c^2\Nr(\alpha x+\beta y)\leq \frac{c^2}{4}(D_1+D_2).
\end{equation}
We now bound $c$. By \eqref{basic}, we have that
\[
\frac{1}{2}D_1D_2D_3 \leq 4p^2 = \det(\gen{x,y,z}) \leq D_1D_2\Nr(u).
\]
It follows that $D_3 \leq 2\Nr(u)$. Furthermore, we observe that
\[
c^2 \Nr(u) = \Nr(s) = \Nr(xy - \tfrac{1}{2} \Tr(xy)) \leq \Nr(xy)=D_1D_2.
\]
Hence
\begin{equation}\label{oneside}
D_3 \leq \frac{2}{c^2}D_1D_2.
\end{equation}
On the other hand, by~\eqref{basic} and~\eqref{Sim1}, we obtain
\[
\frac{9}{64}D_1D_2 < \frac{3}{4}p < \frac{4p^2}{D_1D_2} \leq D_3.
\]
Combined with \eqref{oneside}, this gives $c^2 < 128/9< 15$. As $c\in\Z$, this implies that $c^2\leq 9$. Therefore, from \eqref{boundbyc}, we obtain
\[
\Nr(ax + by) \leq \frac{9}{4}(D_1+D_2) < \frac{9}{4}(15+\frac{16p/3}{15}) < p,
\]
where the last two inequalities follow from \eqref{Sim1}, \eqref{Sim2} and \eqref{Sim3}. However, since $\Nr(a'x+b'y)=\Nr(ax+by)$ from \eqref{axby}, we obtain by Theorem $2'$ of \cite{Kaneko} that $ax+by=\pm (a'x+b'y)$, and so $z=\pm z'$ as desired.
\end{proof}

% REMOVED FOR SPACE REASONS
%\begin{remark}
%As we will discuss following Lemma~\ref{sublattice} in Remark~\ref{bottleneck}, we can loosen the bound \eqref{Sim1} to $D_1D_2 < mp$ for any $m < \sqrt{32}$. This maintains the fact that \eqref{oneside} and \eqref{changem} imply the inequality $c^2 < 16$. It can be readily checked that for all $m<\sqrt{32}$, all other inequalities in the proof hold for sufficiently large $p$. We could even increase $m$ above $\sqrt{32}$, thus losing the implication that $c^2 < 16$, but changing \eqref{Sim2} to $M(m) \leq D_1$, for $M(m)$ sufficiently large and dependent on $m$. The important fact to check would be that the corresponding version of \eqref{axby2} is still bounded above by $p$ for our choice of $m$. However, since $D_1D_2 < 16p/3$ and $15 \leq D_1$ are already the limiting factors in Lemma~\ref{firstmin}, for simplicity we have chosen to take $m=16/3$ here too.
%\end{remark}

Finally, Lemma~\ref{Lem} completes the proof of Theorem~\ref{theorem1}.

%---- SECTION MOVED FROM HERE -----

\section{Algorithm to associate elliptic curves to maximal orders}\label{Algorithm}

In this section we consider the following problem: Given a maximal order $\OO \subset  B_p$, to compute an elliptic curve $E / \F_{p^2}$ such that $\End(E) \cong \OO$.
Our approach is to determine $j(E)$ using Hilbert class polynomials.
We give a general method, but we are only able to prove that this method terminates under the condition~\eqref{maininequality} (e.g., when $\sqrt{-p} \in \OO$, or equivalently, $j(E) \in \F_p$).

Let $H_{D}(X) \in \F_p[X]$ be the reduction modulo $p$ of the Hilbert class polynomial of discriminant $D<0$  (see Section~13 of Cox~\cite{Cox}). We recall that $H_{D}(X)\in\Z[X]$ is the polynomial whose roots are the $j$-invariants of the elliptic curves over $\C$ possessing the quadratic order $\OO_{D} = \Z[\frac{1}{2}(D+\sqrt{D})]$ as their endomorphism ring.

As mentioned in the introduction, if $\sqrt{-p} \in \OO$ then $\OO$ can be written in a canonical form given by Ibukiyama~\cite{Ibukiyama}.
For example, when $p \equiv 1 \pmod{4}$ then there exists a prime $q \equiv 3 \pmod{8}$ and an integer $r$ such that $ q \mid (r^2 + p)$ and such that $\OO $ is isomorphic to an order with $\Z$-basis $\{ 1, (1 + j)/2, i(1+j)/2, (r + i)j/q \}$ in the quaternion algebra defined by $i^2 = -p, j^2 = -q$ and $ij = -ji$.
In the case $p \equiv 3 \pmod{4}$ there are two such families of orders.
Note that $j(E) \in \F_p$ is a root of either $H_{-p}(X)$ or $H_{-4p}(X)$, and is also a root of either $H_{-q}(X)$ or $H_{-4q}(X)$. When $q$ is small this already gives an efficient way to determine $j(E)$, however we cannot assume that $q$ is always small in Ibukiyama's result.

The idea of the algorithm is to use lattice algorithms (basis reduction or enumeration) to find several small norms $d_1, d_2, \dots , d_n$ of primitive elements in $\OO^T$, and to note that $(X - j(E))$ is a factor of $\gcd( H_{-d_1}(X), H_{-d_2}(X), \dots,  H_{-d_n}(X) )$.
To see this note that if $\psi \in \OO^T$ has norm $d$ then $\psi^2 = -d$. By the remark before Lemma~\ref{Lem}, either $(1 + \psi)/2$ or $\psi/2$ lies in $\OO$. Hence $\OO$ contains $\Z[ (d + \sqrt{-d})/2 ]$ and so $j(\OO)$ is a root of $H_{-d}(X)$.

Theorem~\ref{maintheorem} shows that if \eqref{maininequality} holds, then the algorithm is guaranteed to terminate within a bounded time. By Lemma~\ref{lem:conditions}, condition \eqref{maininequality} holds in particular when $j(\OO) \in \F_p$.

The above sketch is made precise in Theorem~\ref{multroot} and Algorithm~1 below. We examine the termination and correctness of Algorithm~1 in the subsequent discussion, and analyse the running time of each specific sub-algorithm in Section~\ref{sec:5.1}.
Some examples of the use of the method are given in Section~\ref{examples}.

We remark that if $p$ is small, then we may identify $j(\OO)$ through exhaustive search. Thus we make the implicit assumption that $p$ is sufficiently large (concretely $p > 286$) so we may use Theorem~\ref{maintheorem}. Furthermore, we recall that the case when $\OO$ has units other than $\pm 1$ is trivial (see beginning of Section~\ref{sec:background}). In the following theorem, the cases $d=3$ and $d=4$ would have corresponded to non-trivial units of $\OO$ when $j(\OO)=1728$ and $j(\OO)=0$ respectively.

%Furthermore, if $\OO$ has a unit (element of norm $1$) other than $\pm 1$, then it is known that $j(\OO) \in \{ 0, 1728 \}$. Precisely, $j(\OO) =0$ if there is a unit of (multiplicative) order $3$ or $6$, and  $j(\OO)=1728$ if there is a unit of order $4$. So these cases pose no problems in identifying $j(\OO)$. In the following theorem, the cases $d=3$ and $d=4$ would have corresponded to non-trivial units of $\OO$ when $j(\OO)=1728$ and $j(\OO)=0$ respectively.

\begin{theorem} \label{multroot}
Assume that $\OO$ has no units other than $\pm 1$. Then $d > 4$ is represented optimally by $\OO^T$ with optimal multiplicity $m$ if and only if $j(\OO)$ appears as a root of $H_{-d}(X) \in \F_p [X]$ with multiplicity $\varepsilon m$, where $\varepsilon = 1$ or $2$ according to whether $p$ is inert or ramified in $\Q(\sqrt{-d})$, i.e., $p$ does not divide or does divide the discriminant $\Delta_{\Q(\sqrt{-d})}$ respectively.
\end{theorem}

\begin{proof}
This can be viewed as a special case of Lemma 3.2 of Elkies et al.~\cite{Elkies}, where the maximal order has no non-trivial units, and so the equivalence class of any optimal embedding $i$ is simply $i$ itself.
We may assume $p$ is inert or ramified because if $p$ splits then the roots of $H_{-d}(X)$ correspond to ordinary elliptic curves.
\end{proof}

We will use Theorem~\ref{multroot} to distinguish orders that have different optimal multiplicities for some integer $d_{n}$.
We use derivatives to achieve this; recall that if a polynomial $p(X)$ over a field $\F$ has $x_0\in\F$ as a root with multiplicity $m \geq 1$, then it holds that $p'(X)$ has $x_0$ as a root with multiplicity $m-1$.

\vspace{3mm}

\noindent {\bf Algorithm~1}
%\vspace{3mm}

\noindent Input: Prime $p$ and a $\Z$-basis of a maximal order $\OO \subset B_p$.

%\vspace{3mm}

\noindent Output: Minimal polynomial of $j$-invariant(s) $j(\OO) \in \F_{p^2}$ such that $\End(E(j(\OO)))=\OO$.
%\vspace{3mm}

\noindent  Procedure:

\begin{enumerate}
\item \label{initial} If $\OO$ has a unit other than $\pm 1$, output the polynomial corresponding to $j(\OO) = 0$ or $j(\OO) = 1728$ accordingly (see discussion before Theorem~\ref{multroot}) and terminate. Otherwise construct a $\Z$-basis of the sublattice $\OO^T$, run lattice reduction/enumeration on the basis, and set $n=1$,  $k = 0$, $\CCC=0$ and $G(X) = 0$.

\item \label{next} Compute $y_n\in\OO^T$ such that $y_n$ is primitive (so $y_n \ne 0$) and $y_n \neq \pm y_i$ for all $1 \leq i < n$, and such that $\Nr(y_n)$ is minimal over all such possible $y_n$.

\item \label{gcd} Set $d_n = \Nr(y_n)$. 
If $p$ divides $\Delta_{\Q(\sqrt{-d_n})}$, set $\varepsilon = 2$, otherwise set $\varepsilon = 1$. 
If $d_n=d_{n-1}$ set $k=k+\varepsilon$, otherwise set $k=\varepsilon - 1$.
If $\varepsilon = 2$ and $k=1$, set $G(X) = \gcd(G(X), H_{-d_n}(X), H'_{-d_n}(X))\in\F_p[X]$.
Otherwise set $G(X) = \gcd(G(X), H^{(k)}_{-d_n}(X))\in\F_p[X]$, where $H^{(k)}_{-d_n}(X)$ is the $k$-th derivative of $H_{-d_n}(X)$, and $H^{(0)}_{-d_n}(X)=H_{-d_n}(X)$.

\item \label{output} If $G(X)$ is either linear, or quadratic and irreducible over $\F_p$, output $G(X)$ and terminate. If $\CCC=1$, or if $n=2$, $15 \leq d_1$ and $d_1d_2 < 16p/3$, proceed to Step~\ref{C}. Otherwise set $n=n+1$ and return to Step~\ref{next}.

\item \label{C} If $n=2$, set $\CCC=1$, $n=3$ and $y_3 = y_1 \pm y_2$, where $+/-$ is chosen to minimize $\Nr(y_3)$. If $n=3$, set $n=4$ and $y_4 = y_1 \pm y_2$, such that $y_4 \neq y_3$. If $n=4$, set $n=5$ and find $y_5$ outside the sublattice $\gen{y_1,y_2}$ such that $\Nr(y_5)$ is minimal. Return to Step~\ref{gcd}.
\end{enumerate}

\vspace{3mm}

If the condition~\eqref{maininequality} holds (e.g., if $j( \OO) \in \F_p $) then the algorithm terminates.
Furthermore, in this case we only need to consider $n \leq 5$ (this is the reason for the addition of Step~\ref{C}, which otherwise seems completely unmotivated).

We hope that the algorithm terminates in all cases, but we do not have a proof of this (see discussion in the following paragraph). We note that since $d_1$ in Step~\ref{next} is simply the first successive minimum of $\OO^T$, it must satisfy $d_1 < p$ (otherwise we contradict~\eqref{basic}). Hence by Theorem $2'$ of Kaneko~\cite{Kaneko} (namely, that if there are two different embeddings of $\Z[ (d  + \sqrt{d})/2 ]$ into $\OO$ then $d^2 \ge p^2$) and Theorem~\ref{multroot} above, $H_{-d_1}(X)$ is square-free, and hence so is $G(X)$ after the first iteration of Step~\ref{gcd}.
Along with Theorem~\ref{multroot}, this implies that if it terminates, Algorithm~1 does compute the correct minimal polynomial of $j(\OO)$. The reason for taking the derivative in Step~\ref{gcd} is to take into account the case of multiple roots of $H_{-d_n}(X)$, i.e., when $\theta_{\OO^T}(d_n) \geq 2$, or when $p$ divides the discriminant of $\Q(\sqrt{-d_n})$.

Let us temporarily stop the algorithm for some $n>0$ just after Step~\ref{gcd}, and for simplicity, let us assume that $d_{n-1} \neq d_n$. Consider the polynomial $G(X)$. One of its roots (or two in the case of a conjugate pair) will be the desired $j$-invariant $j(\OO)$. If $j(\OO')$ is another root of $G(X)$, what can we say about the associated maximal order $\OO'$? It must be the case that $\theta'_{\OO^T}(k) \leq \theta'_{\OO'^T}(k)$ for all integers $k \leq d_{n-1}$, in which case we say that $\OO'^T$ {\it optimally dominates} $\OO^T$ {\it up to} $d_{n-1}$. If the algorithm never terminates, it is clear then that there must exist a maximal order $\OO'$ such that $\theta'_{\OO^T}(k) \leq \theta'_{\OO'^T}(k)$ for all $k > 0$, i.e., $\OO'^T$ optimally dominates $\OO^T$ up to $b$ for all $b > 0$, in which case we simply say that $\OO'^T$ {\it optimally dominates} $\OO^T$. So the question of whether Algorithm~1 terminates, and if so, under what running time, is equivalent to the question of whether there exists another maximal order $\OO' \subset B_p$, of a different type to $\OO$, such that $\OO'^T$ optimally dominates $\OO^T$, and if not, what is a bound $b>0$ such that $\OO'^T$ does not optimally dominate $\OO^T$ up to $b$ for all other maximal orders $\OO' \subset  B_p$. We suspect that such an order $\OO'$ does not exist and we propose the following two conjectures.

\begin{conjecture}\label{exists}
There do not exist two maximal orders $\OO,\OO'\subset   B_p$ of different types such that $\OO'^T$ optimally dominates $\OO^T$.
\end{conjecture}

\begin{conjecture}\label{b}
There exists a bound $b=O(p)$ such that for all maximal orders $\OO,\OO'\subset B_p$ of different types, $\OO'^T$ does not optimally dominate $\OO^T$ up to $b$.
\end{conjecture}

\subsection{Analysis of running time}\label{sec:5.1}

We discuss each step of Algorithm~1 individually. We now assume that~\eqref{maininequality} holds and so we know the algorithm terminates.

Step~\ref{initial} and~\ref{next}: The units of $\OO$ are easily found and so the first part of Step~\ref{initial} poses no problem. We observe that $\OO^T=\gen{v_1,v_2,v_3}$ is a $3$-dimensional sublattice of $\OO=\gen{1,u_1,u_2,u_3}$, where $\{v_1,v_2,v_3\}$ can be given explicitly in terms of $\{u_1,u_2,u_3\}$ as in the discussion preceding Lemma~\ref{determinant}. Hence constructing $\OO^T$ in Step~\ref{initial} and searching for short elements $y_n$ of $\OO^T$ in Step~\ref{next} can be done using standard lattice techniques in polynomial time.

Step~\ref{gcd}: Several algorithms exist to compute $H_{-d_n}(X)$, see, for example, Belding, Br{\" o}ker, Enge and Lauter~\cite{Belding} or Sutherland~\cite{Sutherland}. Under the generalised Riemann hypothesis, $H_{-d_n}(X)$ can be calculated in $\tilde{O}( d_n )$ time.
It is known that $\deg(H_{-d_n}(X)) = h_{-d_n}$, the class number of the imaginary quadratic order $\Z[\frac{1}{2}(d_n+\sqrt {-d_n})]$.

To compute the $\gcd$ of $G(X)$ and $H_{-d_n}(X)$ in Step~\ref{gcd} when $\deg(G(x)) \ge 1$ we use a quasi-linear method (see, for example, Section 8.9 of Aho et al. \cite{Aho} or Section~11.1 of~\cite{GaGe}). 
Hence, this stage can be done in $\tilde{O}( h_{-d_n} )$ operations in $\F_p$. 
By Lemma 1 of~\cite{Belding}, we have $h_{-d_n} = O(\sqrt{{d_n}}\log {d_n})$, and so the $\gcd$ computation can be done in $O(d_n^{0.5+\varepsilon})$ field operations.

As a result, we see that the limiting step of Algorithm~1 is the calculation of $H_{-d_n}(X)$, which is bounded by $O(d_n^{1+\varepsilon})$. 
By~\eqref{D-generalbounds}, $D_1, D_2, D_3, \Nr(x+y)$ and $\Nr( x-y)$ are all $O( p )$.
It follows that the running time of Algorithm~1 under condition~\eqref{maininequality} is $O(p^{1+\varepsilon})$ field operations.
We note that under~\eqref{maininequality}, we have by~\eqref{D-bounds} that $D_3 > 3p/4$, so we do not expect to have a faster running time if $D_3$ is required.

More generally, if we no longer assume~\eqref{maininequality}, then the $O( p )$ bound on the norms is Conjecture~\ref{b}. 
To analyse the running time of Algorithm~1 in the general case under Conjecture~\ref{b}, we must bound the number of elements of $\OO^T$ with norm less than $b$, i.e., the largest possible value for $n$ in the algorithm (under condition~\eqref{maininequality} we knew this was $n \leq 5$). 
Let $B_r$ be the ball of radius $r$ in $\R^m$ centered at the origin.
A special case of a result due to Henk~\cite{Henk} is that for any lattice $L$ of $\R^m$ with successive minima $D_1,D_2,\ldots, D_m$, it holds that
\[
\#(L \cap B_r) < 2^{m-1}\prod_{i=1}^{m} \left\lfloor \frac{2r}{\sqrt{D_i}}  + 1 \right\rfloor.
\]
Equation~\eqref{basic} implies $D_3 \ge D_2 \ge 2 \sqrt{p}$, so taking $r=\sqrt{b}$ and $b=O(p)$ gives $\#\{x \in \OO^T \mid \Nr(x) < b\} = O(p^{0.5})$.
This means $n \leq O(p^{0.5})$ and, since $d_i < b = O(p)$ for every $1 \leq i \leq n$ in Step~\ref{gcd}, we obtain a running time of $O(p^{1.5+\varepsilon})$ field operations under Conjecture~\ref{b}.

We remark that by itself Conjecture~\ref{exists} is equivalent to the fact that Algorithm~1 halts for every maximal order $\OO$, but it does not allow us to make any statements about its running time. We hence stress that even termination is conjectural without assuming~\eqref{maininequality} or Conjecture~\ref{exists}.

Lemma~\ref{lem:conditions} tells us that $D_1D_2 < 16p/3$ will always hold when $j(\OO)\in\F_p$. As remarked before, by finding an element $\pi \in \OO$ such that $\pi^2 = -p$, we can tell if we are in the case when $j(\OO)\in\F_p$. Hence, provided that it is computationally easier to determine the existence of such an element than to run the algorithm until $n = 5$, we could determine before running the algorithm if indeed $j(\OO)\in\F_p$. Unfortunately, the number of supersingular $j$-invariants in $\F_{p^2}$ is approximately $p/12$, and of these, only $H(-4p)  = O(\sqrt{p}\log p)$ lie in $\F_p$, where $H(-4p)$ is the Hurwitz class number (see, for example, 
Theorem~14.18 of Cox~\cite{Cox}).
This shows that for a random maximal order $\OO\subset B_p$, we definitely do not expect that $j(\OO)\in \F_p$.
On the other hand, if the order $\OO$ is input using the format in Ibukiyama~\cite{Ibukiyama} then we know $\sqrt{-p} \in \OO$ and so $j( \OO) \in \F_p$.

\subsection{Algorithm to match all supersingular $j$-invariants with all maximal orders}

In \cite{Cervino}, Cervi\~no proposed an algorithm that, given a prime $p$, associates to every supersingular $j$-invariant of $\F_{p^2}$ the corresponding maximal order type of $B_p$. This is different to Algorithm~1 in that it deals with all $j$-invariants at once. 
Cervi\~no states that his algorithm has running time $\tilde{O}( p^{2.5} )$ operations but no explanation for this is given in the paper and, as far as we can tell, the algorithm he presents is actually at best  $\tilde{O}( p^{4} )$ field operations.
To recall, Cervi\~no computes, on one side, a list of all $O(p)$ maximal orders and, for each such order $\OO$, the set $\Gamma( \OO ) = \{ ( \Tr( \alpha ), \Nr( \alpha ) ) : \alpha \in \OO, \Nr( \alpha ) = O( p ) \}$.
On the other side he computes a list of all $O( p ) $ supersingular elliptic curves and, for each, the set $\Delta(E) = \{ ( \Tr( \phi ), \deg( \phi ) ) : \phi \in \End( E), \deg( \phi ) = O(p) \}$.
Computing $\Gamma( \OO )$ appears to require running over the $O( p^2 )$ elements in the $\Z$-module of rank $4$, hence requiring $O( p^2 )$ work, at best.
Cervi\~no suggests to compute $\Delta(E)$ using V{\' e}lu's  formulae (and this seems to require $O( p^{3+\varepsilon} )$ field operations), but one can probably improve this to $O( p^{2  + \varepsilon} )$ operations using evaluated modular polynomials $\Phi_d( j(E), y ) \in \F_p[x]$, computed using Sutherland's algorithm~\cite{Suth2012}.
Hence, it seems possible to improve Cervi\~no's algorithm so that it requires $O( p^{3 + \varepsilon} )$ field operations.

We propose an alternative algorithm to solve this problem. The main idea of our method is to replace isogeny computations, for a very large set of isogenies, by gcds of Hilbert class polynomials. This leads to a complexity of $O(p^{2.5+\varepsilon})$ field operations.

If we consider the sub-problem of matching supersingular curves over $\F_p$ with their maximal orders, it seems that Cervi\~no's algorithm can be adapted to handle this case with complexity $O( p^{2.5 + \varepsilon} )$ field operations.
Our method for this case has the improved complexity $O( p^{1.5 + \varepsilon} )$.
Note that, as would be expected, the complexities in both cases are just the complexity from Section~\ref{sec:5.1} multiplied by the number of choices for $\OO$.

Cervi\~no's proof that the algorithm halts within a bounded running time uses a result of Schiemann (Theorems~4.4 and~4.5 of \cite{Schiemann}) that two ternary forms with equal theta series are equivalent. In our case, this translates to: if $\OO^T$ and $\OO'^T$ represent the same integers with the same multiplicity, then it follows that $\OO^T\sim\OO'^T$, and hence by Lemma~\ref{Lem}, we have that $\OO$ and $\OO'$ are of the same type.
Furthermore, Schiemann gives a bound $b$ in terms of the successive minima $D_1$, $D_2$ and $D_3$ of $\OO^T$, such that if $\OO^T$ and $\OO'^T$ represent all integers $k \leq b$ with the same multiplicity, then indeed $\OO$ and $\OO'$ are of the same type. For our purposes we may take $b = 3D_3$, which gives $b \le 6p$ using~\eqref{D-generalbounds}, although much better bounds are given in Schiemann's general result.

It is not difficult to see that $\OO^T$ and $\OO'^T$ represent the same integers with the same multiplicity if and only if they optimally represent the same integers with the same optimal multiplicity. This is because every representation $x\in\OO^T$ of $k\in\Z$ can be decomposed uniquely as $x=cy$, where $y\in\OO^T$ is optimal and $c$ is a positive integer. More specifically, we have the following:

\begin{lem}\label{thetaseries}
For any bound $b >0$, it holds that $\theta_{\OO^T}(k) = \theta_{\OO'^T}(k)$ for all $k \leq b$ if and only if $\theta'_{\OO^T}(k) = \theta'_{\OO'^T}(k)$ for all $k \leq b$.
\end{lem}

We now present our alternative to Cervi\~no's algorithm in the general case of all supersingular curves over $\F_{p^2}$.

\vspace{3mm}

\noindent {\bf Algorithm~2}
%\vspace{3mm}

\noindent Input: Prime $p$.
%\vspace{3mm}

\noindent Output: The list of pairs $(\OO_1,K_1(X)), \ldots, (\OO_{t_p}, K_{t_p}(X))$, where $t_p$ is the type number of $B_p$, and for all $1\leq i \leq t_p$, $\OO_i$ are representatives of the distinct maximal order types of $B_p$, and $K_i(X)$ is the minimal polynomial of the supersingular $j$-invariant(s) $j(\OO_i)$.
%\vspace{3mm}

\noindent Procedure:

\begin{enumerate}
\item \label{initial2}
For all $1 \leq i \leq t_p$, compute a $\Z$-basis of $\OO_i$ and $\OO_i^T$, run lattice reduction/enumeration on the bases to compute the successive minima $D_1^i$, $D_2^i$ and $D_3^i$ of $\OO_i^T$, and set $\DDD_i=0$.

\item \label{next2} For every $1 \leq i \leq t_p$ run Algorithm~1 on $\OO_i$ up until it either halts normally or until we reach $n$ such that $d_n > 6p$. If Algorithm~1 halted normally, let $K_i(X)$ be its output, store the pair $(\OO_i,K_i(X))$, and set $\DDD_i=1$. Otherwise let $G_i(X)$ be the current polynomial after Step~\ref{gcd} of Algorithm~1, and store the pair $(\OO_i,G_i(X))$.

\item \label{gcd2} For all $1 \leq i,j \leq t_p$ such that $\DDD_i = 0$ and $\DDD_j=1$, remove from $G_i(X)$ all common factors with $K_j(X)$. If $G_i(X)$ is now either linear, or quadratic and irreducible over $\F_p$, let $K_i(X) = G_i(X)$ and store the pair $(\OO_i,K_i(X))$ and set $\DDD_i=1$.

\item \label{output2} Repeat Step~\ref{gcd2} until $\DDD_i=1$ for all $1\leq i \leq t_p$. Output the list of pairs
\[
(\OO_1,K_1(X)), \ldots, (\OO_{t_p}, K_{t_p}(X)).
\]
\end{enumerate}

\vspace{3mm}

The correctness of Algorithm~2 is guaranteed by the correctness of Algorithm~1. Furthermore Algorithm~2 is always guaranteed to halt, which may seem surprising given that we do not know if the same is true for Algorithm~1 in the general case. To see that Algorithm~2 does always halt, we define a transitive order $\preceq$ on the set of maximal order types as follows: $\OO_i \preceq \OO_k$ if and only if $\OO_k$ optimally dominates $\OO_i$ up to $6p$  (meaning that $\theta'_{\OO_i^T}(m) \le \theta'_{\OO_k^T}(m)$ for all $1 \le m \le 6p$).

We observe that if $\OO_i \preceq \OO_k$ and $\OO_k \preceq \OO_i$, then both orders $\OO_i$ and $\OO_k$ represent the same integers up to $6p$ with the same optimal multiplicity, and so it follows by Schiemann~\cite{Schiemann} and Lemma~\ref{thetaseries} that they are of the same type, i.e., $\OO_i=\OO_k$. Hence $\preceq$ is a partial order on the set of maximal order types $\{\OO_1, \OO_2, \dots , \OO_{t_p}\}$.

Now consider that we have just finished Step~\ref{next2} of Algorithm~2 and consider $1\leq i\leq t_p$ such that $\DDD_i = 0$ (if $\DDD_i = 1$ for all $1\leq i\leq t_p$ then the algorithm clearly terminates without even performing Step~\ref{gcd2}). WLOG assume $i = 1$. From the discussion following Algorithm~1, we know $G_1(X)$ is square-free and so before performing Step~\ref{gcd2} we can write
\[
G_1(X) = (X - j_1) (X - j_2) \cdots (X - j_k),
\]
where the $j$-invariants $j_1, j_2, \ldots, j_k$ are all distinct and represent at least two different maximal orders i.e., we don't have $k=1$, nor do we have $k=2$ and $j_1, j_2$ form a conjugate pair.
WLOG assume that $\OO(j_1) = \OO_1$ i.e., $j_1$ is the correct $j$-invariant associated with $\OO_1$, and likewise that $\OO(j_2) = \OO_2, \OO(j_3) = \OO_3$, etc..

Since the roots $j_2, j_3, \ldots, j_k$ were not removed from $G_1(X)$ when we ran Step~\ref{next2}, this implies that $\OO_2, \OO_3, \ldots, \OO_k$ all optimally dominate $\OO_1$ up to $6p$, i.e., we have $\OO_1 \prec \OO_i$ (meaning that $\OO_1 \preceq \OO_i$ and $\OO_1 \not\cong \OO_i$) for all $1 \leq i \leq k$.

Assume now that $\DDD_1$ never becomes $1$ after any number of repetitions of Step~\ref{gcd2}. This implies that one of $\DDD_2, \DDD_3, \ldots, \DDD_k$ always remains $0$ as well, since otherwise the roots $j_2, j_3, \ldots, j_k$ would ultimately be removed from $G_1(X)$ with enough repetitions of Step~\ref{gcd2}. WLOG assume that $\DDD_2$ always remains $0$. But now the same argument applies to $\DDD_2$, and there must exist another index $1 \leq i \leq t_p$ such that $\OO_2 \prec \OO_i$ and that $\DDD_i$ always remains $0$.

Hence we can find an ascending chain $\OO_1 \prec \OO_2 \prec \OO_i \prec \ldots$ such that $\DDD_1, \DDD_2, \DDD_i, \ldots$ all remain $0$. However every ascending chain clearly has an upper bound, so let us take $\OO_1 \prec \OO_2 \prec \OO_i \prec \ldots \prec \OO_n$, where $\DDD_1, \DDD_2, \DDD_i, \ldots, \DDD_n$ all remain $0$, and such that we cannot find another order $\OO_m$ such that $\OO_n \prec \OO_m$ and $\DDD_m$ always remains $0$. But this implies that $\DDD_n$ ultimately becomes $1$ after a finite number of repetitions of Step~\ref{gcd2}, which clearly leads to a contradiction. It follows that eventually $\DDD_i$ becomes $1$ for every $1 \leq i \leq t_p$, which is equivalent to Algorithm~2 halting with the correct output.

To analyze the running time of Algorithm~2, we start by looking at Step~\ref{next2}. By the same argument as in the analysis of the running time of Algorithm~1 (there under Conjecture~\ref{b}) we conclude that Step~\ref{next2} can be done in time $O(p^{1.5+\varepsilon})$ for every $1 \leq i \leq t_p$. Since $t_p$ is approximately $p/24$, Step~\ref{next2} can be done overall in time $O(p^{2.5+\varepsilon})$.

By earlier discussion and results from Cervi\~no~\cite{Cervino}, Steps~\ref{initial2},~\ref{gcd2} and~\ref{output2} can be done within this running time also.
Hence the overall complexity of Algorithm~2 is $O(p^{2.5 + \varepsilon})$. We stress that in contrast to Algorithm~1, Algorithm~2 is guaranteed to always halt within this running time irrespective of Conjectures~\ref{exists} and~\ref{b}.

Finally, we remark that Algorithm~2 can be restricted to the case when $j(\OO) \in \F_p$. It is possible to enumerate in Step~\ref{initial2} the maximal order types $\OO_1, \OO_2, \ldots, \OO_{H(-4p)}$ whose $j$-invariants lie in $\F_p$ in $O(p^{0.5+\varepsilon})$ field operations~\cite{Koh12}. From the analysis of Algorithm~1 under condition~\eqref{maininequality}, we know that Step~\ref{next2} of Algorithm~2 can be done in time $O(p^{1+\varepsilon})$ for every $1 \leq i \leq H(-4p)$. Since $H(-4p) = O(p^{0.5+\varepsilon})$, this leads to a complexity of $O(p^{1.5+\varepsilon})$ in this restricted case.

\section{Two Examples}\label{examples}

We demonstrate two examples of how Algorithm 1 runs, which were both constructed using Magma~\cite{Magma}.

\begin{example}
Let $p=61$. The quaternion algebra $B_{61}$ is spanned by $\{1,i,j,k\}$ where $i^2=-61,j^2=-7$ and $k=ij=-ji$.

It can be checked that
\[
\OO=\Z + \Z\left(\frac{1}{2}+\frac{1}{2}j\right) + \Z\left(-\frac{1}{2}-\frac{1}{14}j+\frac{1}{7}k\right) + \Z\left(-\frac{1}{2}+\frac{1}{2}i-\frac{3}{14}j-\frac{1}{14}k\right)
\]
is a maximal order of $B_{61}$.

We construct $\OO^T$ and find that its shortest element is $y_1 = j$. We set $d_1 =\Nr(y_1)= 7$, and
\[
G(X) = H_{-d_1}(X) = H_{-7}(X)= X-41 \in \F_{61}[X].
\]
We conclude that the $j$-invariant associated to the maximal order $\OO$ is $j(\OO)=41 \in \F_p$.
\end{example}

\begin{example}
Let $p=20063$. The quaternion algebra $B_{20063}$ is spanned by $\{1,i,j,k\}$ where $i^2=-20063,j^2=-1$ and $k=ij=-ij$. We take $\OO$ as the maximal order in $B_{20063}$ with $\Z$-basis
\begin{align*}
\OO &= \Z\left(\frac{1}{2}+\frac{1}{16}j+\frac{13615}{16}k\right) + \Z\left(\frac{1}{512}i+\frac{151}{4096}j+\frac{1109113}{4096}k\right) \\
&+ \Z\left(\frac{1}{8}j+\frac{13615}{8}k\right) + 2048\Z k.
\end{align*}

We construct $\OO^T$ and begin searching through its short elements. We find
\[
y_1=\frac{11}{64}i-\frac{8323}{512}j+\frac{51}{512}k,
\]
which gives
\[
d_1 = \Nr(y_1)=1056,
\]
and
\[
G_1(X)=H_{-d_1}(X)=H_{-1056}(X) \in \F_{20063}[X],
\]
where $\deg(H_{-1056}(X))=16$.

Next we find
\[
y_2=\frac{67}{256}i+\frac{52101}{2048}j-\frac{85}{2048}k,
\]
which gives
\[
d_2 = \Nr(y_2)=2056,
\]
and
\[
G_2(X)=\gcd(G_1(X),H_{-2056}(X)) = X^3+8728X^2+8070X+5035 \in \F_{20063}[X],
\]
where $\deg(H_{-2056}(X))=16$.

Next we find
\[
y_3=\frac{23}{256}i+\frac{85393}{2048}j-\frac{289}{2048}k
\]
which gives
\[
d_3=\Nr(y_3)=2300,
\]
and
\[
G_3(X)=\gcd(G_2(X),H_{-2300}(X))=X^2+2748X+6627=(X-\alpha)(X-\overline{\alpha}) \in \F_{20063}[X],
\]
where $\deg(H_{-2300}(X))=18$ and $\alpha,\overline{\alpha}$ form a conjugate pair.

Hence we conclude that $\OO$ corresponds to a conjugate pair of supersingular $j$-invariants, $j(\OO)=\alpha,\overline{\alpha}$ with minimal polynomial $X^2+2748X+6627$ over $\F_{20063}$.
\end{example}

\subsection*{Acknowledgements}

We are very grateful to David Kohel for answering our questions about quaternion algebras and to John Voight for his helpful discussions.

\appendix

\section{Proof of Theorem~\ref{maintheorem}}\label{sec:thmMainProof}

We now present the proof of Theorem~\ref{maintheorem}. As with Theorem~\ref{theorem1}, the first step is to take appropriate sublattices $\gen{ x, y }$ in $\OO^T$ and $\gen{x',y'}$ in $\OO'^T$ and to show that $\gen{ x, y }$ and $\gen{x',y'}$ are isometric.  This is done by first proving that $D_1' = D_1$ and then that $D_2' = D_2$.
The final stage of the proof is to extend to the full lattices $\OO^T$ and $\OO'^T$.

\subsection{Proving that $\gen{ x, y }$ and $\gen{x',y'}$ are isometric}

Since $x$ and $y$ represent the first two successive minima of $\OO^T$, we have $\Nr(x+y) = \Nr(x)+\Nr(y)+\Tr(x\overline y) \geq \Nr(y)$ and likewise $\Nr(x-y) = \Nr(x)+\Nr(y)-\Tr(x\overline y) \geq \Nr(y)$. It follows that $|\Tr(x\overline y)| \leq \Nr(x) = D_1$ as otherwise one of these two inequalities would not hold. We hence have $\Tr(x\overline y) = \mu D_1$ for some $|\mu|\leq 1$, and WLOG take $-1\leq\mu\leq 0$ (as otherwise we swap the sign of either $x$ or $y$). Similarly we will let $\Tr(x' \overline{y'}) = \lambda D_1'$ with $-1\leq\lambda\leq 0$.

\begin{lem} \label{lem:triv-cases}
Let notation be as above.
Then $-1 < \mu, \lambda \le 0$ and $D_1 \ne D_2$.
\end{lem}

\begin{proof}
We first show that the cases $\mu=-1$ and $\lambda = -1$ are impossible. If $\mu = -1$, then  $\Nr(y) = \Nr(x+y)$. Hence $D_2$ would have two different optimal representations in $\OO^T$, and so Theorem $2'$ of Kaneko \cite{Kaneko} implies that $D_2^2 \geq p^2$. As $D_3 \geq D_2$,~\eqref{mainconditions} would imply that $D_1D_2D_3 \geq 15 p^2 > 8p^2$, which contradicts \eqref{basic}. So $\mu=-1$ indeed is impossible. Similarly if $\lambda = -1$, then $D_2'^2 \geq p^2$. By~\eqref{D2ineq} this would imply $D_2\geq p$, and we again reach the same contradiction. The same application of Kaneko's result shows that $D_1 \neq D_2$.
\end{proof}

As shown in Section~\ref{sec:OT}, $p \mid 4 D_1 D_2 - T^2$.
On page 853 of \cite{Kaneko}, Kaneko obtains this result by writing $\alpha_1 = (x + D_1)/2$ and $\alpha_2 = (y+D_2)/2$, defining $s = \Tr(\alpha_1 \alpha_2)$, and considering the quantity $(D_1D_2 - (2s-D_1D_2)^2)/4$.  Note that $2s - D_1 D_2 = \Tr( x y) / 2 = -T/2$ so this is just $(D_1 D_2 - (T/2)^2 )/4$.
It is straightforward to verify that
\[
s = -\frac{\mu}{4}D_1 + \frac{1}{2}D_1D_2.
\]
Substituting this value for $s$, we find that $4p$ divides $D_1(D_2-\mu^2D_1/4)$. The same result applies to $\OO'^T$ (which is actually where we will use it), and so defining $M := D_1'(D_2'-\lambda^2D_1'/4)$, it follows that
\begin{equation}\label{M}
4p \leq M.
\end{equation}
We remark that the above with~\eqref{basic} gives
\begin{equation}
  \label{D-generalbounds}
    4p \le D_1 D_2
    \quad \text{ and } \quad
     D_3 \leq 2p,
\end{equation}
and in particular under conditions~\eqref{mainconditions},
\begin{equation}
  \label{D-bounds}
    4p \le D_1 D_2 < \tfrac{16}{3} p
    \quad \text{ and } \quad
     \tfrac{3}{4} p < D_3 \leq 2p.
\end{equation}

We now begin to prove some technical lemmas. The following lemma will only be used in the context of maximal orders, but we remark that it can be readily generalized to all $2$-dimensional lattices.

\begin{lem} \label{nextmin}
Under the condition $\mu, \lambda \in (-1,0]$, $x+y$ is the next shortest element of $\gen{x,y}$ after $\pm y$ which is not in $\gen{x}$, and likewise $x'+y'$ is the next shortest element of $\gen{x',y'}$ after $\pm y'$ which is not in $\gen{x'}$.
\end{lem}

\begin{proof}
We need to check that $\Nr(ax+by)=a^2D_1+b^2D_2+ab\mu D_1$ will always exceed $\Nr(x+y)=D_1+D_2+\mu D_1$ for $a,b \in \Z$ unless $a=b=\pm 1$.

The case $a=0$ is trivial since $x+y$ is strictly shorter than $2y$. So we assume that $a \geq 1$ (otherwise swap $a,b$ with $-a,-b$ everywhere).

We have $a^2D_1+b^2D_2+ab\mu D_1 = aD_1(a+b\mu)+b^2D_2$. So if $a+ b\mu \geq 0$ then for $|b|\geq 2$ we have
\[
aD_1(a+b\mu)+b^2D_2 \geq b^2D_2 = D_2 + D_2 (b^2-1) > \Nr(x+y).
\]
And if $a + b\mu <0$ then $0 < a < b$ and $-ab < a(a+b\mu)$, and so for $b \geq 2$ we have
\[
aD_1(a+b\mu)+b^2D_2 > bD_2(b-a) \geq 2D_2 \geq \Nr(x+y).
\]

Hence we are left with the case $|b|=1$. We now no longer assume $a \geq 1$, but instead WLOG assume $b=1$. It is clear that for $|a| \geq 2$ it holds that
\[
D_2 + a(a+\mu)D_1 \geq D_2 + 2D_1 > D_2+D_1 \geq \Nr(x+y).
\]
Hence we only have to consider $|a|=1$ and clearly we have $\Nr(x-y) \geq \Nr(x+y)$ (with equality only if $\mu=0$), and so indeed $x+y$ is the next shortest element of $\gen{x,y}$ after $\pm y$ which not in $\gen{x}$ as claimed. The same exact argument applies to $x'+y'$.
\end{proof}

The following lemma is the first of two technical lemmas, being Lemmas~\ref{firstmin} and~\ref{sublattice}. In these lemmas we require bounds on $D_1$, $D_1D_2$, and sometimes on $p$ which we explicitly state. The bounds required by the following Lemma~\ref{firstmin} are the strictest and, unlike in Lemma~\ref{sublattice}, we have not yet found a way to loosen them. If the bound on  $D_1D_2$ in the following lemma can be loosened, then the restriction imposed in Theorem~\ref{maintheorem} can be loosened as well.

\begin{lem} \label{firstmin}
Let notation be as in Notation~\ref{defn:notation}. Assume $D_1$ and $D_2$ are both represented optimally by $\OO'^T$. Then $D_1=D_1'$ provided that
\begin{equation}\label{F1}
D_1D_2 < \frac{16}{3}p \text{ and }
\end{equation}
\begin{equation}\label{F2}
8 \leq D_1.
\end{equation}
\end{lem}

\begin{proof}
We first prove that the vectors of $\OO'^T$ that optimally represent $D_1$ and $D_2$ lie in $\gen{x',y'}$. We recall that since $D_1$ and $D_2$ are represented optimally by $\OO'$, we have~\eqref{D2ineq}. By~\eqref{F1} this implies $D_1'D_2' < 16p/3$, and so from~\eqref{basic} we have
\[
\frac{3}{4}p< \frac{4p^2}{D_1'D_2'} \leq D_3'.
\]
Since the norm of the shortest element in $\OO'^T$ outside $\gen{x',y'}$ is $D_3'$, if $D_2$ is represented outside $\gen{x',y'}$ then $3p/4 < D_3' \leq D_2$ and hence
\[
D_1 < \frac{16p}{3D_2} < \frac{64}{9} < 8,
\]
which contradicts~\eqref{F2}. So $D_2$ cannot be represented outside $\gen{x',y'}$. Clearly $D_1$ cannot be represented outside $\gen{x',y'}$ either.

We now assume $D_1=\Nr(ax'+by')$ with $b\neq 0$. This implies in particular that $D_2' \leq D_1$, and so by \eqref{F1} we have
\begin{equation}\label{sqrtp}
D_2' < \frac{4}{\sqrt 3}\sqrt p.
\end{equation}

From Lemma~\ref{nextmin}, we know that $x'+y'$ is the next shortest element after $\pm y'$ in $\gen{x',y'}\setminus\gen{x'}$, and we recall from Lemma~\ref{lem:triv-cases} that $\lambda \in (-1,0]$ and $D_1\neq D_2$. The latter implies that $D_1$ and $D_2$ must have different optimal representations in $\OO'^T$, and so it follows that $\Nr(x'+y') = D_2'+(1+\lambda)D_1' \leq D_2$. 
Combined with $D_2' \leq D_1$, we have that
\begin{equation}\label{ineq1}
D_2'(D_2'+(1+\lambda)D_1') \leq D_1D_2 < \frac{16}{3}p.
\end{equation}

We recall the definition $M=D_1'(D_2'- \lambda^2D_1'/4)$ and define
\[
K = \frac{1}{1+\lambda}\left(\frac{16p}{3D_2'}-D_2'\right).
\]
We will show that $M < 4p$ under the constraints
\[
D_1' \leq \min\{D_2', K\},
\]
and this will be a contradiction to \eqref{M}.

We consider two cases depending on whether or not $D_2' \leq K$. Note that this happens exactly when $(D_2')^2 (2 + \lambda ) \le 16p/3$.

First note that $M$ is maximised when $D_1'$ is as large as possible.
In the case $(D_2')^2 (2 + \lambda ) \le 16p/3$ this means $D_1' = D_2'$ and so
\[
  M \leq D_2'^2 \left(\frac{4-\lambda^2}{4}\right) \leq \frac{16}{3}p\frac{1}{\lambda+2} \left(\frac{4-\lambda^2}{4}\right) < 4p.
\]
In the case $(D_2')^2 (2 + \lambda ) > 16p/3$ we take $D_1' = K$.
Writing $\gamma = (D_2')^2$ we have
\begin{equation}\label{2ndcase}
  M\leq \frac{1}{4(1+\lambda)^2 \gamma}\left(\frac{16}{3}p- \gamma \right)\left( \gamma (\lambda+2)^2-\lambda^2\frac{16}{3}p\right).
\end{equation}
The RHS of \eqref{2ndcase} is subject to the constraints $\gamma =D_2'^2\leq D_1^2< 16p/3$ (which comes from \eqref{sqrtp}) and $16p <3(\lambda+2) \gamma$. It is then routine to verify that the RHS of \eqref{2ndcase} is maximized when $\gamma$ is minimal, i.e., $\gamma = \frac{16}{3(\lambda+2)}p$ (a simple way to verify this is to compute the partial derivative of the RHS of \eqref{2ndcase} with respect to $\gamma$ and observe that it is negative when $16|\lambda| p< 3(\lambda + 2)\gamma $). Substituting $\gamma = \frac{16}{3(\lambda+2)}p$ into the RHS of \eqref{2ndcase} reduces it to $4(2-\lambda)p/3$, which for $\lambda \in(-1,0]$ is always less that $4p$.

Hence, in both cases, we obtain that $M < 4p$, which contradicts \eqref{M}. In conclusion, if $D_1$ and $D_2$ are both represented optimally by $\OO'^T$ with $D_1 = \Nr(ax'+by')$, then we must have $b = 0$ and it follows that $a=1$ and $D_1=D_1'$.
\end{proof}

\begin{lem} \label{sublattice}
Let notation be as in Notation~\ref{defn:notation}. Assume $D_1=D_1'$ and that $D_2$, $\Nr(x+y)$ and $\Nr(x-y)$ are all represented optimally by $\OO'^T$. Then $x\sim x'$, $y\sim y'$ and $x+y\sim x'+y'$ (from which it will follow that $\gen{x,y}\sim \gen{x',y'}$ by Lemma~\ref{sim}) provided that
\begin{equation}\label{S1}
D_1D_2 < 7 p,
\end{equation}
\begin{equation}\label{S2}
15 \leq D_1, \text{ and }
\end{equation}
\begin{equation}\label{S3}
286 < p.
\end{equation}
\end{lem}

\begin{proof}
In light of Lemma~\ref{cond}, it suffices to prove that $D_2=D_2'$ and $\Nr(x+y)=\Nr(x'+y')$ since all vectors in question have zero trace.

Recall that $\Nr( x+y) = (1 + \mu) D_1 + D_2$ and $\Nr( x' + y') = (1 + \lambda) D_1' + D_2'$ where $-1 < \mu, \lambda \le 0$.
To avoid trivial cases later on, we first prove that $\mu,\lambda \neq 0$. From Lemma~\ref{nextmin}, we know that $\Nr(x+y) \leq \Nr(x-y)$, and if equality held, then $\Nr(x+y)=\Nr(x-y)=D_1+D_2$, which by Theorem $2'$ of \cite{Kaneko} implies that $(D_1+D_2)^2 \geq p^2$ and so $D_1 + D_2 \geq p$. As $D_3 \geq D_2$, this in turn implies
\[
D_1D_2D_3 \geq D_1 D_2^2 \geq D_1(p-D_1)^2 > 8p^2,
\]
where the last inequality is true for $15 \leq D_1 < \sqrt{7 p}$ and $p$ in \eqref{S3}, which contradicts \eqref{basic}. As a result $\Nr(x+y) < \Nr(x-y)$ which is indeed equivalent to $\mu \in (-1, 0)$. The same exact argument (keeping in mind that $D_1'=D_1$) shows that $\lambda\neq 0$, and so indeed we have that $\mu,\lambda \in (-1,0)$.

Now we prove that the vectors in $\OO'^T$ which represent $\Nr(x)$, $\Nr(y)$, $\Nr(x+y)$ and $\Nr(x-y)$ all lie in $\gen{x',y'}$. The longest of these vectors, $x-y$, has norm  $(1 - \mu)D_1 + D_2 \leq 2D_1+D_2$, which from \eqref{S1} and \eqref{S2}, is bounded by $2D_1+D_2 < 30 + 7p/15$. On the other hand, from $D_2' \leq D_2$ we obtain $D_1'D_2' < 7 p$, and hence we have from \eqref{basic} that
\[
\frac{4p}{7} < \frac{4p^2}{D_1'D_2'} \leq D_3'.
\]
This implies that for $p$ in \eqref{S3} we have
\begin{equation}\label{mlimit}
2D_1+D_2 \leq 30 + \frac{7p}{15} < \frac{4}{7}p < D_3'.
\end{equation}
Since $D_3'$ is the norm of the shortest element of $\OO'^T$ outside $\gen{x',y'}$, we see that none of $D_1$, $D_2$, $\Nr(x+y)$, $\Nr(x-y)$ can be represented outside $\gen{x',y'}$.

Hence assume $D_2=\Nr(ax'+by')$. Remarking that $a(a+b\lambda) \geq -\left(\lambda b/2\right)^2$, and recalling that $D_1=D_1'$ by assumption, we obtain
\[
D_2=a^2D_1'+b^2D_2'+ab\lambda D_1'=aD_1'(a+b\lambda)+b^2 D_2'\geq b^2D_2'-\left(\frac{\lambda b}{2}\right)^2 D_1,
\]
which implies $D_2' \leq D_2/b^2 + \lambda^2D_1/4$. Hence by \eqref{S1}, for $|b|\geq 2$ we have
\[
M =  D_1'D_2'-\frac{\lambda^2}{4}D_1'^2 \leq D_1\left(\frac{1}{b^2}D_2 +\frac{\lambda^2}{4}D_1\right)-\frac{\lambda^2}{4}D_1^2 = \frac{D_1D_2}{b^2} < 4p,
\]
which contradicts \eqref{M}, and so we must have $|b| = 1$. WLOG (changing the sign of $a$ if necessary), we can take $b=1$.

Now let $\Nr(x+y)=(1 + \mu) D_1+D_2 =\Nr(cx'+dy') = c^2 D_1'+d^2D_2'+cd\lambda D_1'$. Remarking as before that  $c(c+d\lambda) \geq -\left(\lambda d/2\right)^2$, we obtain
\[
\Nr(x+y) = D_1(1+\mu)+D_2 \geq d^2D_2'-\left(\frac{\lambda d}{2}\right)^2D_1'.
\]
This with \eqref{S1} implies that, for $|d|\geq 2$, we have
\[
M = D_1'D_2'-\frac{\lambda^2}{4}D_1'^2 \leq D_1\frac{D_1(1+\mu)+D_2+\frac{\lambda^2d^2}{4}D_1}{d^2}-\frac{\lambda^2}{4}D_1^2\leq \frac{2D_1D_2}{d^2} < 4p,
\]
which again contradicts \eqref{M}, and so we must have $|d|=1$. WLOG (changing the sign of $c$ if necessary), we can take $d=1$.

Since $D_1=D_1'$ and $b=d=1$, we have
\begin{equation}\label{D2}
D_2=a(a+\lambda)D_1+D_2' \text{ and}
\end{equation}
\begin{equation}\label{D1+D2}
D_1(1+\mu)+D_2=c(c+\lambda)D_1+D_2'.
\end{equation}

We observe that $a\neq c$ since otherwise $\mu =-1$, which is impossible from before. So subtracting \eqref{D2} from \eqref{D1+D2}, factorizing and dividing, gives us
\begin{equation}\label{diph1}
\frac{1+\mu}{c-a}=a+c+\lambda.
\end{equation}
We observe that if $a=0$ then $1+\mu=c(c+\lambda)$, where the LHS is in $(0,1)$, which implies from the RHS that $c=1$. But this implies that $D_2 = D_2'$ and $\Nr(x+y) = \Nr(x'+y')$ as desired, and we conclude by Lemma~\ref{cond}.

So we assume now that $a \neq 0$. We note that if $a=1$, then \eqref{diph1} becomes $1+\mu = c(c+\lambda) - 1 -\lambda$, from which we see that the only possible solution (since the LHS is again in $(0,1)$) is $c=-1$ and $\lambda = -(1+\mu)/2 \in (-1/2,0)$.

We now claim that
\begin{equation}\label{D_2}
D_2 < \frac{7}{4}D_2'.
\end{equation}
Indeed, if this was not the case, by \eqref{S1} we would have
\[
M \leq D_1'D_2'\leq \frac{4}{7}D_1D_2<4p,
\]
which contradicts \eqref{M}.

Now \eqref{D_2} and \eqref{D2} imply that $a(a+\lambda)D_1+D_2'=D_2 \leq 7D_2'/4$. We remark that $a(a+\lambda)> 0$ for all integers $a\neq 0$. Hence we have
\begin{equation}\label{justanotherbound}
D_1 \leq \frac{3D_2'}{4a(a+\lambda)}.
\end{equation} 

Now let $\Nr(x-y)=(1 - \mu) D_1+D_2 =\Nr(ex'+fy') = e^2 D_1'+f^2D_2'+ef\lambda D_1'$. We remark that $e^2+\lambda ef \geq -\left(\lambda f/2\right)^2$, and so with \eqref{justanotherbound}, we have
\[
D_2 \geq f^2D_2'+\left(-\left(\frac{\lambda f}{2}\right)^2-(1-\mu)\right)D_1 \geq D_2' \left( f^2 - \frac{3}{4a(a+\lambda)}\left( 1-\mu+\frac{\lambda^2f^2}{4}\right)\right)
\]
\begin{equation}\label{ineq2}
= D_2' \left( f^2\left(1 -\frac{3\lambda^2}{16a(a+\lambda)}\right) - \frac{3(1-\mu)}{4a(a+\lambda)}\right).
\end{equation}
We observe that for all $\lambda \in (-1,0)$ and $a \in\Z$, with $a\neq 0$, and with $\lambda \in (-1/2,0)$ when $a=1$, it holds that
\[
   \delta = 1 -\frac{3\lambda^2}{16a(a+\lambda)} > 0.
\]
Hence for all $|f|\geq 2$, it holds that
\begin{equation}\label{ineq3}
  D_2 \geq D_2' \left( 4 \delta -  \frac{3(1-\mu)}{4a(a+\lambda)}  \right)
          \geq D_2' \left(4-\frac{3(1-\mu+\lambda^2)}{4a(a+\lambda)}\right).
\end{equation}
By separating into the cases $a \leq -2, a= -1, a = 1$ and $a\geq 2$, it can be readily checked that for $\lambda,\mu \in (-1,0)$ and $a\in\Z$, with $a\neq 0$, and with $\lambda = -(1+\mu)/2$ when $a=1$, it holds that
\[
\frac{1-\mu+\lambda^2}{a(a+\lambda)} \leq \frac{5}{2},
\]
with equality only in the case that $a=1$ and $\mu=0$, $\lambda = -1/2$. As a result,
\[
  D_2 \geq D_2'\left(4-\frac{15}{8}\right) > \frac{7}{4}D_2',
\]
which contradicts \eqref{D_2}. We conclude that $|f| \geq 2$ is impossible, and hence WLOG, we take $f=1$.

We now have
\begin{equation}\label{D1-D2}
D_1(1-\mu)+D_2=eD_1(e+\lambda)+D_2'.
\end{equation}
Viewing \eqref{D2} and \eqref{D1-D2}, we observe that $e\neq a$, as otherwise we would have $\mu =1$, which is impossible. Hence subtracting \eqref{D2} from \eqref{D1-D2} we obtain
\begin{equation}\label{diph2}
\frac{1-\mu}{e-a}=a+e+\lambda.
\end{equation}
Viewing this in conjunction with \eqref{diph1}, we wish to find the possible solutions to \eqref{diph1} and \eqref{diph2} with $a,c,e \in \Z$, $a\neq 0$, and $\lambda,\mu\in(-1, 0)$.

We observe that if $e-a=1$ then the LHS of \eqref{diph2} is in $(1,2)$, which implies $a+e=2$. However this implies $2e=3$, which is impossible. If $e-a=-1$, then the LHS of \eqref{diph2} is in $(-2,-1)$, which implies $a+e=-1$. However this implies $e=-1$ and $a=0$, and we already saw that $a=0$ implied the result of the theorem.

So we are only left to consider the case that $|e-a| \geq 2$. If $e-a\geq 2$, then the LHS of \eqref{diph2} is in $(0,1)$, which implies that $a+e=1$. If $e-a\leq -2$ then the LHS of \eqref{diph2} is in $(-1,0)$, which implies that $a+e=0$. Exactly the same reasoning applies to \eqref{diph1} with $e$ replaced by $c$. As a result, we have the following implications:
\[
c-a \geq 2 \Longrightarrow a+c=1 \Longrightarrow 1-2a \geq 2 \Longrightarrow a < 0,
\]
\[
c-a \leq -2 \Longrightarrow a+c=0 \Longrightarrow -2a \leq -2 \Longrightarrow a > 0,
\]
\[
e-a \geq 2 \Longrightarrow a+e=1 \Longrightarrow 1-2a \geq 2 \Longrightarrow a < 0,
\]
\[
e-a \leq -2 \Longrightarrow a+e=0 \Longrightarrow -2a \leq -2 \Longrightarrow a > 0,
\]
with other values for $c-a$ and $e-a$ being impossible.

From this we see that if $a > 0$, then the only possibility for $e$ and $c$ is $e=c=-a$, and if $a < 0$, then the only possibility is $e=c=1-a$. In either case we obtain $e=c$. But together with \eqref{diph1} and \eqref{diph2}, this implies that $1+ \mu = 1 - \mu$ and so $\mu=0$, which we excluded earlier.

We conclude that the only possible solution to $D_2 = \Nr(ax'+by')$, $\Nr(x+y) = \Nr(cx'+dy')$ and $\Nr(x-y) = \Nr(ex'+fy')$ is $a=0$, $b=1$, $c=1$, $d=1$, $e=-1$, $f=1$ (and the corresponding negative solutions if we wish to change signs). This implies by Lemma~\ref{cond} that $y\sim y'$ and $x+y \sim x'+y'$ as desired.
\end{proof}

%\begin{remark}\label{bottleneck}
%The bottleneck in the proof was in \eqref{mlimit}, which we used to show that $\Nr(x-y)$ could not be represented outside $\gen{x',y'}$. In this inequality, we see that the bound \eqref{S1} could be replaced by $D_1D_2 < mp$, for any $m < \sqrt{60}$. An appropriate change to \eqref{S3} would then also need to be made. It can easily be verified that all other inequalities in the proof will hold for all such $m$, so indeed \eqref{mlimit} is the limiting factor in the proof. Since $D_1D_2 < 16p/3$ is already the limiting factor in Lemma~\ref{firstmin}, we used the bound $7p$ in Lemma~\ref{sublattice} for sake of simplicity.
%\end{remark}

%\begin{remark}
%Departing from the assumption $D_1=D_1'$, it becomes easier to show our desired result, in the sense that the bound $D_1D_2 < 16p/3$ from Lemma~\ref{firstmin} can be loosened to $D_1D_2 < mp$ for $m$ as in Remark~\ref{bottleneck}. This seems a little strange, as intuitively, $D_1=D_1'$ should be the most obvious condition which should be satisfied. Furthermore, the proof of Lemma~\ref{firstmin} did not use the assumption that $\Nr(x+y)$ and $\Nr(x-y)$ are also optimally represented by $\OO'^T$. This suggests that there could be an alternative way to prove that $D_1=D_1'$ than in Lemma~\ref{firstmin}, and hence to loosen the conditions imposed by Theorem~\ref{maintheorem}.
%\end{remark}

\subsection{Completing the proof}

We have shown that $\gen{ x, y }$ and $\gen{x',y'}$ are isometric.
Hence, by Lemma~\ref{sim}, we can conjugate $\OO$ by an appropriate element $c\in  B_p$ and hence assume that $\gen{ x, y } = \gen{x',y'}$. 
It remains to deal with $D_3$.

After conjugation, we have that $\OO^T = \gen{ x, y, z}$ and $\OO'^T = \gen{ x, y, z'}$ where $\Nr( z ) = D_3$ and $\Nr( z' ) = D_3'$.
Since $z, z' \not\in \gen{x,y}$ and $\theta'_{\OO^T}(D_3) \leq \theta'_{\OO'^T}(D_3)$ it follows that $D_3' \le D_3$.  The next result shows that we may assume $D_3' = D_3$, in which case the proof will follow from the argument used to prove Theorem~\ref{theorem1}.

\begin{lem} \label{gen}
Let notation be as in Notation~\ref{defn:notation}. Suppose that $\gen{ x, y } = \gen{x',y'}$. Suppose furthermore that there exists $w \in \OO'^T$, $w\notin \gen{x,y}$, such that $\Nr(w) = D_3$. It holds that $w = \pm z'$.
%$\OO'^T=\gen{x,y,w}$, i.e., $\{x,y,w\}$ forms a $\Z$-basis for $\OO'^T$.
\end{lem}

Lemma~\ref{gen} is true for any two $3$-dimensional lattices of equal determinant defined over a space with a positive bilinear form, but we will only use it in the context given above.

\begin{proof}[of Lemma~\ref{gen}]
As in Lemma~\ref{simorders} we let $u$ and $u'$ be the projections of $z$ and $z'$ to $\gen{x,y}^\perp$, and deduce that $u' = u$.

Now we observe from Lemma~\ref{ternarylattice} that
\[
   \det(\OO^T)\leq\det(\gen{x,y})\Nr(z) \leq D_1D_2D_3\leq 2\det(\OO^T),
\]
from which it follows that
\begin{equation}\label{firstineqgen}
  \Nr(z)=D_3 \leq \frac{2\det(\OO^T)}{\det(\gen{x,y})}.
\end{equation}
On the other hand, as $D_3$ is represented by $w\in \OO'^T = \gen{x,y,z'}$ outside of $\gen{x,y}$, we have that $w=ax+by+cz'$ for some $a,b,c\in\Z$, $c\neq 0$. Therefore
\[
D_3=\Nr(w) = \Nr(ax+by+cz') \geq c^2\Nr(u')=c^2\frac{\det(\OO^T)}{\det(\gen{x,y})},
\]
where the last equality comes from \eqref{projections}. Combined with \eqref{firstineqgen}, this implies that $c=\pm 1$, and the conclusion follows.
\end{proof}

\begin{proof}[of Theorem~\ref{maintheorem}]
Assume that $D_1$, $D_2$, $\Nr(x+y)$, $\Nr(x-y)$ and $D_3$ are all optimally represented in $\OO'^T$ and that $\theta'_{\OO^T}(D_3) \leq \theta'_{\OO'^T}(D_3)$. The case $D_1 < 15$ is treated by Lemma~\ref{lem:trivialcase} so we assume conditions~\eqref{mainconditions}. From Lemma~\ref{nextmin}, we know that $D_1'=D_1$. Hence, from Lemma~\ref{sublattice}, we have that $y\sim y'$ and $x+y\sim x'+y'$. By consequence, from Lemma~\ref{sim}, by conjugating $\OO'$ by an appropriate element $c\in B_p$, we can assume that $\gen{x,y}=\gen{x',y'}$. Now, in order that $\theta'_{\OO^T}(D_3) \leq \theta'_{\OO'^T}(D_3)$, we require that $D_3$ is represented in $\OO'^T$ outside of $\gen{x,y}$. Hence, by Lemma~\ref{gen} we may assume that $D_3' = D_3$.
Lemma~\ref{simorders} then implies $\OO^T = \OO'^T$. Lemma~\ref{Lem} implies that $\OO$ and $\OO'$ are of the same type as desired.
This completes the proof of Theorem~\ref{maintheorem}.
\end{proof}

%\affiliationone{I. Chevyrev\\ Mathematical Institute\\ University of Oxford\\ UK. \email{ilya.chevyrev@maths.ox.ac.uk}}

%\affiliationtwo{S. D. Galbraith\\ Mathematics Department\\ University of Auckland\\ New Zealand. \email{S.Galbraith@math.auckland.ac.nz}}

\end{document}